\documentclass{amsart}

\usepackage[all, cmtip]{xy}
\usepackage{amsmath,amsfonts,amssymb,amsthm,pict2e,enumerate,chngcntr,caption}
\usepackage[T1]{fontenc}
\usepackage[citecolor=blue, colorlinks=true, linkcolor=blue, urlcolor=blue]{hyperref}

\theoremstyle{definition}
\newtheorem{defn}{Definition}[section]
\newtheorem{assumption}[defn]{Assumption}
\newtheorem{ex}[defn]{Example}
\newtheorem{remark}[defn]{Remark}

\theoremstyle{plain}
\newtheorem{lemma}[defn]{Lemma}
\newtheorem{theorem}[defn]{Theorem}
\newtheorem{proposition}[defn]{Proposition}
\newtheorem{corollary}[defn]{Corollary}

\newcommand{\Z}{\mathbb{Z}}

\newcommand{\C}{\mathbb{C}}

\newcommand{\R}{\mathbb{R}}
\newcommand{\HH}{\mathbb{H}}
\newcommand{\Proj}{\mathbb{P}}
\newcommand{\WP}{\mathbb{W}\mathbb{P}}

\newcommand{\calL}{\mathcal{L}}

\newcommand{\calM}{\mathcal{M}}
\newcommand{\calN}{\mathcal{N}}
\newcommand{\calA}{\mathcal{A}}

\newcommand{\calX}{\mathcal{X}}
\newcommand{\calY}{\mathcal{Y}}
\DeclareMathOperator{\rank}{rank}
\DeclareMathOperator{\Pic}{Pic}

\DeclareMathOperator{\T}{T}

\DeclareMathOperator{\Id}{Id}

\DeclareMathOperator{\Aut}{Aut}

\DeclareMathOperator{\NS}{NS}

\DeclareMathOperator{\SO}{SO}

\counterwithin{table}{section}
\counterwithin{equation}{section}

\begin{document}

\title{Families of lattice polarized K3 surfaces with monodromy}

\author[C. F. Doran]{Charles F. Doran}
\address{Department of Mathematical and Statistical Sciences, 632 CAB, University of Alberta, Edmonton, Alberta T6G 2G1, Canada}
\email{charles.doran@ualberta.ca}
\thanks{C. F. Doran and A. Y. Novoseltsev were supported by the Natural Sciences and Engineering Resource Council of Canada (NSERC), the Pacific Institute for the Mathematical Sciences, and the McCalla Professorship at the University of Alberta}

\author[A. Harder]{Andrew Harder}
\address{
Department of Mathematical and Statistical Sciences, 632 CAB, University of Alberta, Edmonton, Alberta T6G 2G1, Canada}
\email{aharder@ualberta.ca}
\thanks{A. Harder was supported by an NSERC PGS D scholarship and a University of Alberta Doctoral Recruitment Scholarship}

\author[A.Y. Novoseltsev]{Andrey Y. Novoseltsev}
\address{
Department of Mathematical and Statistical Sciences, 632 CAB, University of Alberta, Edmonton, Alberta T6G 2G1, Canada}
\email{novoselt@ualberta.ca}

\author[A. Thompson]{Alan Thompson}
\address{
Fields Institute, 222 College Street, Toronto, Ontario M5T 3J1, Canada}
\email{amthomps@ualberta.ca}
\thanks{A. Thompson was supported in part by NSERC and in part by a Fields Institute Ontario Postdoctoral Fellowship with funding provided by NSERC and the Ontario Ministry of Training, Colleges
and Universities}

\date{\today}

\begin{abstract} We extend the notion of lattice polarization for K3 surfaces to families over a (not necessarily simply connected) base, in a way that gives control over the action of monodromy on the algebraic cycles, and discuss the uses of this new theory in the study of families of K3 surfaces admitting fibrewise symplectic automorphisms. We then give an application of these ideas to the study of Calabi-Yau threefolds admitting fibrations by lattice polarized K3 surfaces.
\end{abstract}

\maketitle

\tableofcontents

\section{Introduction}

The concept of lattice polarization for K3 surfaces was first introduced by Nikulin \cite{fagkk3s} and further developed by Dolgachev \cite{mslpk3s}. Our aim is to extend this theory to families of K3 surfaces over a (not necessarily simply connected) base, in a way that allows control over the action of monodromy on algebraic cycles.

Our interest in this problem arises from the study of Calabi-Yau threefolds with small Hodge numbers. In their paper \cite{doranmorgan}, Doran and Morgan explicitly classify the possible integral variations of Hodge structure that can underlie a family of Calabi-Yau threefolds over the thrice-punctured sphere $\Proj^1 - \{0,1,\infty\}$ with $h^{2,1} = 1$. Explicit examples, coming from toric geometry, of families realising all but one of these variations of Hodge structure were known at the time of publication of \cite{doranmorgan}, and a family realising the fourteenth and final case was recently constructed in \cite{14thcase}.

One of the main tools used to study the Calabi-Yau threefolds constructed in \cite{14thcase} was the existence of a \emph{torically induced} fibration (i.e. a fibration of the threefold induced by a fibration of the toric ambient space by toric subvarieties) of these threefolds by K3 surfaces polarized by the rank $18$ lattice
\[M := H \oplus E_8 \oplus E_8.\]

K3 surfaces polarized by this lattice have been studied by Clingher, Doran, Lewis and Whitcher \cite{milpk3s}\cite{nfk3smmp} and have a rich geometric structure. In particular, the canonical embedding of the lattice $E_8 \oplus E_8$ into $M$ defines a natural Shioda-Inose structure on them, which in turn defines a canonical Nikulin involution \cite{k3slpn}. The resolved quotient by this involution is a new K3 surface, which may be seen to be a Kummer surface associated to a product of two elliptic curves; its geometry is closely related to that of the original K3 surface.

In \cite{14thcase}, toric geometry was used to show that this Nikulin involution is induced on the $M$-polarized K3 fibres by a global involution of the Calabi-Yau threefold. The resolved quotient by this involution is another Calabi-Yau threefold, which is fibred by Kummer surfaces and has geometric properties closely related to the first. Examination of this second Calabi-Yau threefold was instrumental in proving that the construction in \cite{14thcase} realised the ``missing'' fourteenth variation of Hodge structure from the Doran-Morgan list.

Motivated by the discovery of this K3 fibration and the rich geometry that could be derived from it, we decided to search for similar K3 fibrations on the other threefolds from the Doran-Morgan classification. In a large number of cases (summarized by Theorem \ref{thm:Mnfamilies}), we found fibrations by K3 surfaces polarized by the rank $19$ lattices
\[M_n := H \oplus E_8 \oplus E_8 \oplus \langle -2n \rangle,\]
which contain the lattice $M$ as a sublattice. Many, but not all, of these fibrations are torically induced.

This raises two natural questions: Do the canonical Nikulin involutions on the fibres of these K3 fibrations extend to global symplectic involutions on the Calabi-Yau threefolds? And if they do, what can be said about the geometry of the new Calabi-Yau threefolds obtained as resolved quotients by these involutions?

Both of these questions may be addressed by studying the behaviour of the N\'{e}ron-Severi lattice of  a K3 surface as it varies within a family. Furthermore, in order for this theory to be useful in the study of K3 fibred Calabi-Yau threefolds it should be able to cope with the possibility of monodromy around singular fibres, meaning that we must allow for the case where the base of the family is not simply connected.

To initiate this study, we introduce a new definition of lattice polarization for families of K3 surfaces and develop the basic theory surrounding it. We note that a related notion of lattice polarizability for families of K3 surfaces was introduced by Hosono, Lian, Oguiso and Yau \cite{adck3smt}, who also proved statements about period maps and moduli for such families. However, our definition is more subtle than theirs, given that our goal is to derive precise data about the monodromy of algebraic cycles. The relationship between the definitions is discussed in greater detail in Remark \ref{rem:hloy}.
\medskip

The structure of this paper is as follows. In Section \ref{section:FamK3} we begin with the central definitions of $N$-polarized (Definition \ref{def:polarized}) and $(N,G)$-polarized (Definition \ref{def:polarizable}) families of K3 surfaces, where $N$ is a lattice and $G$ is a finite group. The first is a direct extension of the definition of $N$-polarization for K3 surfaces to families and does not allow for any action of monodromy on the lattice $N$. The second is more subtle: it allows for a nontrivial action of monodromy, but this monodromy is controlled by the group $G$.

The remainder of Section \ref{section:FamK3} proves some basic results about $N$- and $(N,G)$-polarized families of K3 surfaces and their moduli. Of particular importance are Proposition \ref{prop:sympaut} and Corollary \ref{cor:sympaut}, which use this theory to give conditions under which symplectic automorphisms can be extended from individual K3 fibres to entire families of K3 surfaces.

Section \ref{section:siif} expands upon these results, focussing mainly on the case where the symplectic automorphism is a Nikulin involution. The main result of this section is Theorem \ref{thm:involutions}, which shows that the resolved quotient of an $N$-polarized family of K3 surfaces, where $N$ is the N\'{e}ron-Severi lattice of a general fibre, by a Nikulin involution is an $(N',G)$-polarized family of K3 surfaces, where $N'$ is the N\'{e}ron-Severi lattice of a general fibre of the resolved quotient family and $G$ is a finite group.

In Section \ref{genconst} we specialize all of these results to families of $M$-polarized K3 surfaces with their canonical Nikulin involution, which extends globally over the family by Corollary \ref{cor:sympaut}. The resolved quotient family is an $(N',G)$-polarized family of K3 surfaces whose general fibre is a Kummer surface. The first major result of this section, Proposition \ref{prop:generalcaseG}, places bounds on the size of the group $G$.

To improve upon this result, in Section \ref{undokummer} we show that, after proceeding to a finite cover of the base, we may realise these families of Kummer surfaces by applying the Kummer construction fibrewise to a family of Abelian surfaces, a process which we call \emph{undoing the Kummer construction}. As a result of this process we obtain Theorem \ref{thm:dekummer} and Corollary \ref{cor:dekummer}, which enable explicit calculation of the group $G$.

In Section \ref{sect:Mnthreefolds} we further specialize this analysis to families of $M_n$-polarized K3 surfaces, then apply the resulting theory to the study of the Calabi-Yau threefolds from the Doran-Morgan list. The main results here are Theorems \ref{thm:Mnfamilies} and \ref{thm:M1families}, which show that twelve of the fourteen cases from that list admit fibrations by $M_n$-polarized K3 surfaces. In fact, we prove an even stronger result: for $n \geq 2$ these fibrations are in fact pull-backs of special $M_n$-polarized families on the moduli space of $M_n$-polarized K3 surfaces, under the \emph{generalized functional invariant} map, and for $n = 1$ they are pull-backs of a special $2$-parameter $M_1$-polarized family by a closely related map.

We compute the generalized functional invariant maps for all of these fibrations in Sections \ref{sect:14caseapp} and \ref{sect:M1}. We find that they all have a standard form, defining multiple covers of the moduli spaces of $M_n$-polarized K3 surfaces with ramification behaviour determined by a pair of integers $(i,j)$.

Finally, in Section \ref{sect:arith/thin} we use these results to make an interesting observation concerning an open problem related to the Doran-Morgan classification. Recall that each of the threefolds from this classification moves in a one parameter family over the thrice-punctured sphere. Recently there has been a great deal of interest in studying the action of monodromy around the punctures on the third integral cohomology group of the threefolds. This monodromy action defines a Zariski dense subgroup of $\mathrm{Sp}(4,\R)$, which may be either arithmetic or non-arithmetic (more commonly called \emph{thin}). Singh and Venkataramana \cite{acshg}\cite{a4mgacyt} have proved that the monodromy is arithmetic in seven of the fourteen cases from the Doran-Morgan list, and Brav and Thomas \cite{tmsp4} have proved that it is thin in the remaining seven. It is an open problem to find geometric criteria that distinguish between these two cases.

In Theorem \ref{thm:thinarith} we provide a potential solution to this problem: the cases may be distinguished by the values of the pair of integers $(i,j)$ arising from the generalized functional invariants of \emph{torically induced} K3 fibrations on them. Specifically, we find that a case has thin monodromy if and only if neither $i$ nor $j$ is equal to two. This suggests that it may be possible to express the integral monodromy matrices for the families of Calabi-Yau threefolds from the Doran-Morgan list in terms of the families of transcendental cycles for their internal K3-fibrations, and that doing so explicitly may be a good route towards an understanding of the geometric origin of the arithmetic/thin dichotomy.

A different criterion to distinguish the arithmetic and thin cases was recently given by Hofmann and van Straten \cite[Section 6]{smgfis}, using an observation about the integers $m$ and $a$ from \cite[Table 1]{doranmorgan} (which are called $d$ and $k$ in \cite{smgfis}). Furthermore, the discovery of a yet another criterion has been announced in lectures by M. Kontsevich, using a technique involving Lyapunov exponents. Whilst our result does not appear to bear any immediate relation to either of these other results, it is our intention to investigate the links between them in future work.

\subsection{Acknowledgements} A part of this work was completed while A. Thompson was in residence at the Fields Institute Thematic Program on Calabi-Yau Varieties: Arithmetic, Geometry and Physics; he would like to thank the Fields Institute for their support and hospitality.

\section{Families of K3 surfaces}
\label{section:FamK3}
Begin by assuming that $X$ is a K3 surface. The N\'eron-Severi group of divisors modulo homological equivalence on $X$ forms a non-degenerate lattice inside of $H^2(X,\mathbb{Z})$, denoted $\NS(X)$, which is even with signature $(1,\rho -1)$. The lattice of cycles orthogonal to $\NS(X)$ is called the lattice of transcendental cycles on $X$ and is denoted $\T(X)$.

The aim of this section is to develop theoretical tools that will enable us to embark upon a study of the action of monodromy on the N\'eron-Severi group of a fibre in a family of K3 surfaces.

\subsection{Families of lattice polarized K3 surfaces}

We begin with some generalities on families of K3 surfaces. A family of K3 surfaces will be a variety $\mathcal{X}$ and a flat surjective morphism $\pi\colon \mathcal{X} \rightarrow U$ onto some smooth, irreducible, quasiprojective variety $U$ such that for each $p \in U$ the fibre $X_p$ above $p$ is a smooth K3 surface. For simplicity the reader may assume that $U$ has dimension $1$ but our results are valid in arbitrary dimension. We further assume that there is a line bundle $\mathcal{L}$ whose restriction to each fibre of $\pi$ is ample and primitive in $\Pic(X_p)$ for each $p \in U$.

In the analytic topology, there is an integral local system on $U$ given by $R^2 \pi_*\mathbb{Z}$ whose fibre above $u$ is isomorphic to $H^2(X_p,\mathbb{Z})$. The Gauss-Manin connection $\nabla_{\mathrm{GM}}$ is a flat connection on $R^2\pi_*\mathbb{Z} \otimes \mathcal{O}_U$.

The cup-product pairing on $H^2(X_p,\mathbb{Z})$ extends to a bilinear pairing of sheaves
\begin{equation}
\label{eqn:pairing}
\langle \cdot, \cdot \rangle _\mathcal{X} = R^2\pi_*\mathbb{Z} \times R^2\pi_* \mathbb{Z} \rightarrow R^4 \pi_* \mathbb{Z} \cong \mathbb{Z}_U
\end{equation} 
where $\mathbb{Z}_U$ is the constant sheaf on $U$ with $\mathbb{Z}$ coefficients. This form extends naturally to arbitrary sub-rings of $\mathbb{C}$.

There is a Hodge filtration on $R^2\pi_* \mathbb{Z} \otimes \mathcal{O}_U$. In particular $\mathcal{H}^{2,0}_{\mathcal{X}} = F^2 (R^2\pi_* \mathbb{Z} \otimes \mathcal{O}_U)$, and there is a local subsystem of $R^2 \pi_* \mathbb{C}$ which gives rise to $\mathcal{H}^{2,0}_\mathcal{X}$ . Choosing a flat local section of $\mathcal{H}^{2,0}_\mathcal{X}$, which we will call $\omega_\mathcal{X}$, we take the local subsystem of $R^2\pi_* \mathbb{Z}$ which is orthogonal to $\omega_\mathcal{X}$. Since the pairing $\langle \cdot,\cdot \rangle_\mathcal{X}$ is $\mathbb{Z}$ linear and $\omega_\mathcal{X}$ is flat, $\omega^\perp_\mathcal{X}$ is defined globally on $U$. We will call this local subsystem $\mathcal{NS}(\mathcal{X})$. Note that $\mathcal{NS}(\mathcal{X})$ is the Picard sheaf of the flat morphism $\pi$.

We let $\mathcal{T}(\mathcal{X})$ be the integral orthogonal complement of $\mathcal{NS}(\mathcal{X})$. We have an orthogonal direct sum decomposition over $\mathbb{Q}$
$$
R^2\pi_*\mathbb{Q} = (\mathcal{T}(\mathcal{X}) \oplus \mathcal{NS}(\mathcal{X}))\otimes_{\mathbb{Z}_U} \mathbb{Q}_U
$$

Our aim is to use this to study the action of monodromy on the N\'{e}ron-Severi lattice of a general fibre of $\calX$. In order to gain control of this monodromy, we begin by extending the definition of lattice polarization for K3 surfaces to families.

To do this, let $\mathcal{N}$ be a local subsystem of $\mathcal{NS}(\mathcal{X})$ such that for any $p \in U$, the restriction of $\langle\cdot, \cdot \rangle_\mathcal{X}$ to the fibre $\mathcal{N}_p$ over $p$ exhibits $\mathcal{N}_p$ as a non-degenerate integral lattice of signature $(1,n-1)$, which is (non-canonically) isomorphic to a lattice $N$ and embedded into $H^2(X_p,\mathbb{Z})$ as a primitive sublattice containing the Chern class of the ample line bundle $\calL_p$. This allows us to define a na\"{i}ve extension of lattice polarization to families.

\begin{defn} \label{def:polarized}
The family $\mathcal{X}$ is \emph{$N$-polarized} if the local system $\mathcal{N}$ is a trivial local system. 
\end{defn}

Note that any family of K3 surfaces is polarized by the rank one lattice generated by the Chern class of the line bundle $\mathcal{L}$ restricted to each fibre.

Unfortunately, this definition is too rigid for our needs: it is easy to see that for an $N$-polarized family of K3 surfaces, a choice of isomorphism $N \cong \mathcal{N}_{p}$ for any point $p$ determines uniquely an isomorphism $N \cong \mathcal{N}_q$ for any other point $q$ by parallel transport, so this definition does not allow for any action of monodromy on $\calN_q$. We will improve upon this definition in Section \ref{section:monodromyautos}, but in order to do so we first need to develop some general theory.

\subsection{Monodromy of algebraic cycles on K3 surfaces}

In this section we will begin discusing the action of monodromy on the N\'{e}ron-Severi group of a general fibre of $\mathcal{X}$. Let $p$ be a point in $U$ such that the fibre above $p$ has $\NS(X_{p}) \cong \mathcal{NS}(\mathcal{X})_{p}$. Parallel transport along paths in $U$ starting at the base point $p$ gives a monodromy representation of $\pi_1(U,p)$
$$
\rho_\mathcal{X}\colon \pi_1(U,p) \rightarrow \mathrm{O}(H^2(X_{p},\mathbb{Z}))
$$
since we have the pairing in Equation \eqref{eqn:pairing}. Furthermore, $\rho_\mathcal{X}$ restricts to monodromy representations of both $\mathcal{NS}(\mathcal{X})$ and $\mathcal{T}(\mathcal{X})$, written as
$$
\rho_\mathcal{NS}\colon \pi_1(U,p) \rightarrow \mathrm{O}(\NS(X_{p}))
$$
and
$$
\rho_\mathcal{T}\colon \pi_1(U,p) \rightarrow \mathrm{O}(\T(X_{p})).
$$
Similarly for any local subsystem $\mathcal{N}$ of $R^2 \pi_*\mathbb{Z}$, we will denote the associated monodromy representation $\rho_\mathcal{N}$. Note here that if $\mathcal{X}$ is $N$-polarized, then the image of $\rho_\mathcal{N}$ is the trivial subgroup $\Id$.

Now we prove an elementary but useful result concerning the image of $\rho_{\mathcal{NS}}$. Here we let $X$ be a K3 surface. Recall that the lattice $\NS(X)$ is an even lattice of signature $(1,\rank \NS(X) -1)$. For such a lattice $\NS(X)$, there is a set of roots
$$
\Delta_{X} = \{ w \in \NS(X): \langle w,w\rangle = -2 \}.
$$
The Weyl group $W_X$ is the group generated by Picard-Lefschetz reflections across roots in $\Delta_X$. It admits an embedding into the orthogonal group $\mathrm{O}(\NS(X))$. Denote the set of roots in $\Delta_X$ which are dual to the fundamental classes of rational curves by $\Delta_{X}^+$. Then a fundamental domain for the action of $W_{X}$ on $\NS(X)$ is given by the closure of the connected polyhedral cone
$$
K(X) = \{ w \in \NS(X) \otimes \mathbb{R} : \langle w,w\rangle > 0, \langle w,\delta \rangle > 0  \ \mathrm{for}\ \mathrm{all}\  \delta \in \Delta_{X}^+\}.
$$
$K(X)$ is the K\"ahler cone of $X$ \cite[Corollary VIII.3.9]{bpv}.

If we let $\mathrm{O}_+(\NS(X))$ be the subgroup of $ \mathrm{O}(\NS(X))$ which fixes the positive cone in $\NS(X)$ and let $D_X$ be the subgroup of $\mathrm{O}_+(\NS(X))$ which maps $K(X)$ to itself, then we obtain a semidirect product decomposition
$$
\mathrm{O}_+(\NS(X)) = D_{X} \ltimes W_{X}.
$$

Now let $L$ be an ample line bundle on $X$. Then the Chern class of $L$ is contained in $K(X)$.  Define $D_{X}^L$ to be the stabilizer of this Chern class in $D_{X}$. 

\begin{proposition}
\label{prop:monodromy}
Let $\mathcal{X}$ be a family of K3 surfaces and let $X_p$ be a generic fibre of $\mathcal{X}$. Let $\mathcal{L}_p$ be the restriction of the bundle $\mathcal{L}$ on $\mathcal{X}$ to $X_p$. Then the group $D_{X_p}^{\mathcal{L}_p}$ is finite and contains the image of $\rho_\mathcal{NS}$.
\end{proposition}
\begin{proof}
First we show that $D_{X_p}^{\mathcal{L}_p}$ is a finite group. Let $\gamma$ be in $D^{\mathcal{L}_p}_{X_p}$. Then $\gamma$ fixes $\mathcal{L}_p$ by definition. Therefore $\gamma$ acts naturally on $[\mathcal{L}_{p}]^\perp$ and fixes $[\mathcal{L}_{p}]^\perp$ if and only if it fixes all of $\NS(X_{p})$. Since $\mathcal{L}_{p}$ is ample, the orthogonal complement of $[\mathcal{L}_p]$ in $\NS(X_p)$ is negative definite by the Hodge index theorem.

We then recall the fact that $\mathrm{O}(N)$ is finite for any definite lattice $N$, so $D^\mathcal{L}_\mathcal{X}$ is contained in a finite group and thus is itself finite.

To see that $\rho_\mathcal{NS}$ has image contained in $D^{\mathcal{L}_p}_{X_p}$, we recall that $\rho_\mathcal{NS}$ fixes $\mathcal{L}_{p} \in K(X_{p})$ and hence, since the closure of $K(X_p)$ is a fundamental domain for $W_{X_p}$ and the action of $W_{X_p}$ is continuous, $\rho_\mathcal{NS}$ must have image in $D_{X_p}$.
\end{proof}

\subsection{Monodromy and symplectic automorphisms} \label{section:monodromyautos}

We are now almost ready to make a central definition which extends Definition \ref{def:polarized} to cope with the possible action of monodromy on $N$. 

Denote by $N^*$ the dual lattice of $N$. We may embed $N^* \subseteq N \otimes_\mathbb{Z} \mathbb{Q}$ as the sublattice of elements $u$ of $N\otimes_\mathbb{Z} \mathbb{Q}$ such that $\langle u,v\rangle \in \mathbb{Z}$ for all $v \in N$. 

\begin{defn}
The \emph{discriminant lattice} of $N$, which we call $A_N$, is the finite group $N^*/N$ equipped with the bilinear form
$$
b_N\colon A_N \times A_N \rightarrow \mathbb{Q} \bmod \mathbb{Z}.
$$ 
induced by the bilinear form on $N\otimes_\mathbb{Z} \mathbb{Q}$. 
\end{defn}

For each lattice $N$ we may define a map $\alpha_N\colon \mathrm{O}(N) \rightarrow \Aut(A_N)$ where $\Aut(A_N)$ is the group of automorphisms of the finite abelian group $A_N$ which preserve the bilinear form $b_N$. Denote the kernel of $\alpha_N$ by $\mathrm{O}(N)^*$. Then we make the central definition:
 
\begin{defn}
\label{def:polarizable}
Fix an even lattice $N$ with signature $(1, n-1)$ and a subgroup $G$ of $\Aut(A_N)$. Let $\mathcal{X}$ be a family of K3 surfaces and let $X_p$ be a generic fibre of $\mathcal{X}$. Assume that there is local sub-system $\mathcal{N} \subseteq \mathcal{NS}(\mathcal{X})$ which has fibres $\calN_p$ that are isometric to $N$ and are embedded into $H^2(X_p,\mathbb{Z})$ as primitive sublattices containing the Chern class of the ample line bundle $\calL_p$. Then $\mathcal{X}$ is called an \emph{$(N,G)$-polarized family of K3 surfaces} if the restriction of the map $\alpha_N$ to the image of $\rho_{\mathcal{N}}$ is injective and has image inside of $G$.
\end{defn}

One sees that if $\Id$ is the trivial subgroup of $\Aut(A_N)$, then the definition of an $N$-polarized family of K3 surfaces is identical to the definition of a family of $(N,\Id)$-polarized K3 surfaces. We also note that, if $G \subset G'$, then any $(N,G)$-polarized family of K3 surfaces will also be $(N,G')$-polarized. With this in mind, we identify a special class of $(N,G)$-polarized families where the group $G$ is as small as possible.

\begin{defn} An $(N,G)$-polarized family of  K3 surfaces $\calX$ is called \emph{minimally $(N,G)$-polarized} if the composition $\alpha_N \cdot \rho_{\mathcal{N}}$ is surjective onto $G$.\end{defn}

\begin{remark}\label{rem:hloy}
We note that in \cite{adck3smt}, the authors introduce a similar notion of $N$-polarizability for a family of K3 surfaces. A K3 surface $X$ is \emph{$N$-polarizable} in the sense of \cite{adck3smt} if there is a sublattice inside of $\NS(X)$ isomorphic to $N$, but the primitive embedding of $N$ into $\NS(X)$ is only fixed up to automorphism of the K3 lattice $\Lambda_{\mathrm{K3}}$. A family of K3 surfaces is then called $N$-polarizable if each fibre is $N$-polarizable. There is a well-defined period space of $N$ polarizable K3 surfaces $\mathcal{M}_{N}^{\circ}$, so that to any family of $N$-polarizable K3 surfaces there is a well-defined period map.

Our definition is more subtle than this, since our goal is to derive precise data about the monodromy of algebraic cycles. Any $(N,G)$-polarized family of K3 surfaces is $N$-polarizable, but the converse does not hold. In fact, both of the families constructed in Section \ref{section:example} are families of $N$-polarizable K3 surfaces, but only one of them is $(N,G)$-polarized.
\end{remark}

There is a close relationship between $(N,G)$-polarizations and symplectic automorphisms. Recall the following definition:

\begin{defn}
\label{def:sympaut}
Let $X$ be a smooth K3 surface and let $\tau\colon X \rightarrow X$ be an automorphism of $X$. The automorphism $\tau$ is called a {\it symplectic automorphism} if for some (hence any) non-vanishing holomorphic 2-form $\omega$ on $X$, $\tau^* \omega = \omega$. If $\tau$ has order $2$, it is called a symplectic involution of $X$ or a {\it Nikulin involution}.
\end{defn}

Symplectic automorphisms of finite order on K3 surfaces exhibit behaviour similar to translation by a torsion section on an elliptic curve. The quotient of an elliptic curve by some subgroup of $\Pic(E)_{\mathrm{tors}}$ is an isogenous elliptic curve, i.e. an elliptic curve $E'$ such that there is a Hodge isometry $H^1(E,\mathbb{Q}) \cong H^1(E',\mathbb{Q})$. Analogously there is a sense in which the resolved quotient of a K3 surface $X$ by a finite group of symplectic automorphisms is isogenous to $X$: there is a real quadratic extension of $\mathbb{Q}$ under which the Hodge structures on their transcendental lattices are isometric. This will be explained in detail by Proposition \ref{prop:hodge}.

The following is a consequence of the famous Global Torelli Theorem for K3 surfaces \cite{ttastk3}\cite{fagkk3s}. More precisely, it may be seen as a corollary of \cite[Theorem 4.2.3]{iqfaag}.

\begin{theorem}
\label{thm:pss}
The kernel of the map $\alpha_{\NS(X_{p})}\colon D_\mathcal{X}^\mathcal{L} \rightarrow \Aut(A_{\NS(X_{p})})$ is isomorphic to the finite group of symplectic automorphisms of $X_{p}$ which fix $[\mathcal{L}_{p}]$.
\end{theorem}

From this, using Proposition \ref{prop:monodromy}, we obtain:

\begin{corollary}
\label{cor:rephrase}
Let $\mathcal{X}$ be a family of K3 surfaces with generic N\'eron-Severi lattice $N$. The family $\mathcal{X}$ is $(N,G)$-polarized for some $G$ in $\Aut(A_N)$ if and only if there is no $\gamma \in \pi_1(U,p)$ such that $\rho_\mathcal{NS}(\gamma) = \sigma|_{\NS(X_{p})}$ for some symplectic automorphism $\sigma$ of $X_{p}$. 
\end{corollary}

Therefore, a measure of how far a family of K3 surfaces with generic N\'eron-Severi lattice $N$ can be from being $(N,G)$-polarized is given by the size of the group of symplectic automorphisms of a generic $N$-polarized K3 surface. The number of possible finite groups of symplectic automorphisms of a K3 surface is relatively small. Mukai \cite[Theorem 0.3]{fgaks3mg} has shown that such groups are all contained as special subgroups of the Mathieu group $M_{23}$, and in particular Nikulin \cite[Proposition 7.1]{fagkk3s} has shown that an algebraic K3 surface with symplectic automorphism must have N\'eron-Severi rank at least $9$. This gives:

\begin{corollary}
Any family of K3 surfaces with generic N\'eron-Severi group $N$ having $\rank(N) < 9$ is $(N,G)$-polarized for some $G \subset \Aut(A_N)$.
\end{corollary}

We end this subsection with a proposition which determines when a symplectic automorphism on a single K3 surface extends to an automorphism on an entire family of K3 surfaces. This will be useful in Section \ref{section:siif}, when we will further discuss symplectic automorphisms in families. 

\begin{proposition}\label{prop:sympaut}
Let $X_{p}$ be a fibre in $\mathcal{X}$ which satisfies $\mathcal{NS}(\mathcal{X})_{p} \cong \NS(X_{p})$, and let $\tau$ be a symplectic automorphism of $X_{p}$. Then $\tau$ extends to an automorphism of $\mathcal{X}$ if and only if its action on $\NS(X_{p})$ commutes with the image of $\rho_{\mathcal{X}}$.
\end{proposition}
\begin{proof}
Since $\mathcal{X}$ is a proper family of smooth manifolds, Ehresmann's theorem (see, for example, \cite[Section 9.1.1]{htcagI}) implies that there is a local analytic open subset, called $U_0$, about $p \in U$, so that there is a marking on the family of K3 surfaces $\mathcal{X}_{U_0}$ on $U_0$. Therefore \cite[Lemma 4.2]{fagkk3s} and the Global Torelli Theorem \cite[Theorem 2.7']{fagkk3s} shows that $\tau$ extends uniquely to an automorphism on $\mathcal{X}_{U_0}$.

Let $\gamma \in \pi_1(U,p)$, let $\gamma^*\tau$ be the analytic continuation of $\tau$ along $\gamma$, and let $w \in H^2(X_0,\mathbb{Z})$. Then it is easy to see that 
$$
\gamma^* \tau(w) = \rho_{\mathcal{X}}(\gamma)\cdot \tau \cdot (\rho_{\mathcal{X}}(\gamma))^{-1}(w).
$$
Therefore, the action of $\tau$ on $\NS(X_{p})$ commutes with the image of $\rho_{\mathcal{X}}$ if and only if the action of $\gamma^*\tau$ on $\NS(X_p)$ agrees with the action of $\tau$. By the Global Torelli Theorem, this happens if and only if the automorphisms $\tau$ and $\gamma^*\tau$ are the same.
\end{proof}

\begin{corollary}
\label{cor:sympaut}
Let $\mathcal{X} \to U$ be an $N$-polarized family of K3 surfaces and suppose $N \cong \NS(X_p)$ for some fibre $X_p$. If $X_p$ admits a symplectic automorphism $\tau$, then $\tau$ extends to an automorphism of $\mathcal{X}$.
\end{corollary}

\subsection{A non-polarizable example}
\label{section:example}
As we have seen, algebraic monodromy of families of K3 surfaces is intimately related to the existence of symplectic automorphisms. In this section, we will give a simple example which will show how the existence of symplectic automorphisms produces non-polarized families of K3 surfaces.

Let us take the pencil of K3 surfaces mirror (in the sense of \cite{mslpk3s}) to the Fermat pencil of quartics in $\mathbb{P}^3$. We may write these surfaces as a family $\mathcal{X}$ of ADE singular hypersurfaces in $\mathbb{P}^3$:
$$
(x + y + z + w)^4 + t^2 xyzw = 0.
$$
As a non-compact threefold, we may express these as a singular subvariety of 
$$
[x:y:z:w] \times t \in \mathbb{P}^3 \times \mathbb{C}^\times.
$$
This is an $(E_8^2 \oplus H \oplus \langle -4\rangle,\Id)$-polarized family of K3 surfaces. Each fibre admits $A_4$ as a group of symplectic automorphisms acting via even permutations on the coordinates $x,y,z,w$. In particular we have a symplectic involution on each fibre induced by
$$
\sigma\colon [x:y:z:w] \mapsto [y:x:w:z],
$$
which extends to $\mathcal{X}$ by Corollary \ref{cor:sympaut}. We also have an involution on the base, acting via
$$
\eta\colon t \mapsto -t.
$$
Therefore, the fibrewise resolutions of the quotient families $\mathcal{Y}_1 = \widetilde{\mathcal{X}/(\Id \times \eta)}$ and $\mathcal{Y}_2 = \widetilde{\mathcal{X} / (\sigma \times \eta)}$ are fibrewise biregular, but are not biregular as total spaces. More importantly both families have the same holomorphic periods, but the monodromy of $\mathcal{NS}(\mathcal{Y}_1)$ is trivial and the monodromy of $\mathcal{NS}(\mathcal{Y}_2)$ is non-trivial around $0$.

Thus we see that the family $\mathcal{Y}_1$ is $N$-polarized. However, by Corollary \ref{cor:rephrase}, the family $\mathcal{Y}_2$ is not $(N,G)$-polarized for any $G$ since, by construction, monodromy around $0$ acts as a Nikulin involution on $\mathcal{NS}(\mathcal{Y}_2)$.

\begin{remark}
Of course this examples and examples like it reflect directly the general principle that there does not exist a fine moduli scheme of objects which admit automorphisms, and in particular this example itself proves that the period space of K3 surfaces is not a fine moduli space. If one considers instead the moduli stack of polarized K3 surfaces (see \cite{mspk3smc}), then such families are distinguished.
\end{remark}

\subsection{Moduli spaces and period maps}
In the last subsection of this section, we will study the moduli of $(N,G)$-polarized families. We begin by establishing some definitions regarding the period spaces of K3 surfaces; much of this material may be found in greater detail in \cite{mslpk3s}. 

Define the K3 lattice to be the lattice $\Lambda_{\mathrm{K3}} = E_8^2 \oplus H^3$. The space of marked pseudo-ample K3 surfaces is the type IV symmetric domain 
$$
\mathcal{P}_{\mathrm{K3}} = \{ z \in \mathbb{P}(\Lambda_{\mathrm{K3}} \otimes \mathbb{C}) : \langle z,z \rangle = 0, \langle z,\overline{z} \rangle >0 \}.
$$
There is a natural action on $\mathcal{P}_{\mathrm{K3}}$ by the group $\mathrm{O}(\Lambda_{\mathrm{K3}})$. Using terminology of \cite{mslpk3s}, the orbifold quotient
$$
\mathcal{M}_{\mathrm{K3}} := \mathrm{O}(\Lambda_{\mathrm{K3}}) \setminus \mathcal{P}_{\mathrm{K3}}
$$
is called the \emph{period space of K\"ahler K3 surfaces}. 

For any even lattice $N$ of rank $n$ and signature $(1,n-1)$ equipped with a primitive embedding $N \hookrightarrow \Lambda_{K3}$, one may construct a period space of pseudo-ample marked K3 surfaces with $N$-polarization. Let
$$
\mathcal{P}_{N} = \{ z \in \mathbb{P}(N^\perp \otimes \mathbb{C}) : \langle z,z \rangle = 0, \langle z,\overline{z} \rangle >0 \}.
$$
There is a natural embedding
$$
\varphi_{N} \colon \mathcal{P}_N \hookrightarrow \mathcal{P}_{\mathrm{K3}}
$$
where we suppress the dependence upon choice of embedding of $N$ into $\Lambda_{\mathrm{K3}}$. Let
$$
\mathrm{O}(N^\perp) = \{ \gamma|_{N^\perp} : \gamma \in \mathrm{O}(\Lambda_{\mathrm{K3}}) , \gamma(N) \subseteq N \}.
$$
The map $\varphi_{N}$ descends to an embedding
$$
\overline{\varphi_{N}} \colon \mathrm{O}(N^\perp) \setminus \mathcal{P}_N \hookrightarrow \mathrm{O}(\Lambda_{\mathrm{K3}}) \setminus \mathcal{P}_{\mathrm{K3}}.
$$

For each group $G_{N^{\perp}}$ in $\Aut(A_{N^\perp})$, we may construct a finite index subgroup of $ \mathrm{O}(N^\perp)$,
$$
\mathrm{O}(N^\perp,G_{N^{\perp}}) = \{ \gamma|_{N^\perp} \in \mathrm{O}(N^\perp): \alpha_{N^\perp}(\gamma|_{N^\perp}) \in G_{N^{\perp}} \}.
$$
This subgroup is related to $(N,G_N)$-polarized K3 surfaces in the following way. Recall the following standard lattice theoretic fact from \cite{isbfa}.
\begin{proposition}\textup{\cite[Proposition 1.6.1]{isbfa}} \label{prop:Nikulin-Lattices}
Let $N$ be a primitive sublattice of an even unimodular lattice $K$, and let $N^\perp$ be the orthogonal complement of $N$ in $K$. Then
\begin{enumerate}[\textup{(}1\textup{)}]
\item There is a canonical isomorphism $\phi^N$ between the underlying groups $A_N$ and $A_{N^\perp}$ which satisfies 
$$
b_N(a,b) = -b_{N^\perp}(\phi^N(a),\phi^N(b)).
$$
\item If $g$ is an automorphism of $N$ and $g'$ is an automorphism of $N^\perp$, then $g \oplus g'$ is an automorphism of $N \oplus N^\perp$ which extends to an automorphism of $K$ if and only if the induced actions of $g$ on $A_N$ and of $g'$ on $A_{N^\perp}$ are the same under the identification $\phi^N$.
\end{enumerate}
\end{proposition}
Therefore, if a family of K3 surfaces $\mathcal{X}$ is $(N,G_N)$-polarized, then Proposition \ref{prop:Nikulin-Lattices} shows that the transcendental monodromy of $\mathcal{X}$ is in $\mathrm{O}(N^\perp,G_{N^{\perp}})$ where $G_{N^{\perp}}$ is the subgroup of $A_{N^\perp}$ identified with $G_N$ by $\phi^N$.

As a particular example, if $\Id $ is the trivial subgroup of $G_N$ then the family $\calX$ is $N$-polarized and the group $\mathrm{O}(N^\perp,\Id)$ corresponds to the group $\mathrm{O}(N^\perp)^*$. By \cite[Proposition 3.3]{mslpk3s}, we have
$$
\mathrm{O}(N^\perp,\Id) = \mathrm{O}(N^\perp)^* \cong \{ \gamma|_{N^\perp} : \gamma \in \mathrm{O}(\Lambda_{\mathrm{K3}}), \gamma(w) = w\ \mathrm{for}\ \mathrm{all}\ w \in N \} .
$$
In the case where our family is $N$-polarized we will use the notation and language of \cite{mslpk3s}, but adopt the notation introduced above when the group $G_N$ becomes relevant. 

In \cite{mslpk3s}, the space
$$
\mathcal{M}_N = \mathrm{O}(N^\perp)^* \setminus \mathcal{P}_N
$$
is called the \emph{period space of pseudo-ample $N$-polarized K3 surfaces}. Dolgachev  \cite[Remark 3.4]{mslpk3s} shows that for any $N$-polarized family of K3 surfaces $\pi\colon\mathcal{X} \rightarrow U$, there is a period morphism
$$
\Phi_\mathcal{X} \colon U \rightarrow \mathcal{M}_N.
$$

In light of this, define 
$$
\mathcal{M}_{(N,G_N)} := \mathrm{O}(N^\perp,G_{N^{\perp}}) \setminus \mathcal{P}_N.
$$
Note that for $G_N \subseteq G'_N$, there is a natural inclusion $\mathrm{O}(N^\perp,G_{N^{\perp}}) \subseteq \mathrm{O}(N^\perp,G'_{N^{\perp}})$ and therefore there are natural surjective morphisms
$$
\mathcal{M}_{(N,G_N)} \rightarrow \mathcal{M}_{(N,G'_N)}
$$
of degree $[G_N:G'_N]$.

We now take some time to prove the existence of period morphisms associated to the spaces $\mathcal{M}_{(N,G_N)}$.

\begin{theorem} \label{thm:periods}
Let $\calX \to U$ be a family of K3 surfaces. If there is some local subsystem $\mathcal{N} \subseteq \mathcal{NS}(\mathcal{X})$, where $\mathcal{N}$ is fibrewise isomorphic to a lattice $N$ of signature $(1,n-1)$ and $\alpha_N \cdot \rho_{\mathcal{NS}}$ is contained inside of a subgroup $G_N$ of $\Aut(A_N)$, then there a period morphism
$$
\Phi_{(N,G_N)} \colon U \rightarrow \mathcal{M}_{(N,G_N)}.
$$
\end{theorem}

\begin{proof}
Let $\widetilde{U}$ be the simply connected universal covering space of $U$ and $g\colon\widetilde{U} \rightarrow U$ be the canonically associated covering map. Then, since $g^*\mathcal{X}$ is marked, pseudo-ample and $N$-polarized, we have the following diagram
\[\xymatrix{
\widetilde{U} \ar[r] \ar[d]_{g} & \mathcal{P}_N \\
U & \\
}\]

Now we apply Proposition \ref{prop:Nikulin-Lattices}. Since the image of $\alpha_N \cdot \rho_{\mathcal{N}}$ is in $G_N$, the image of $\alpha_{N^\perp} \cdot \rho_{\mathcal{N}^\perp}$ is contained in $G_{N^{\perp}}$ under the identification induced by $\phi^N$. Thus $\rho_{N^\perp}$ is contained in $\mathrm{O}(N^\perp,G_{N^{\perp}})$. 

This allows us to canonically complete the diagram above to a commutative square
\[\xymatrixcolsep{3pc}\xymatrix{
\widetilde{U} \ar[r] \ar[d]_{g} & \mathcal{P}_N \ar[d] \\
U \ar[r]^-{\Phi_{(N,G_N)}} & \mathcal{M}_{(N,G_N)} \\
}\]
as required.
\end{proof}

We note that the assumptions in this proposition are weaker than the assumption that $\calX \to U$ is $(N,G_N)$-polarized, as we do not assume here that the map  $\alpha_N$ is injective on the image of $\rho_{\mathcal{NS}}$. What distinguishes $(N,G_N)$-polarized families of K3 surfaces from the rest is the following observation.

\begin{remark}\label{rmk:alg-trans dictionary}
Let $\mathcal{X} \rightarrow D^*$ be an $(N,G_N)$-polarized family of K3 surfaces over the punctured disc $D^*$, and let $\gamma$ be a generator of $\pi_1(D^*,p)$ and $u \in N \subseteq \NS(X_p)$ with $\overline{u}$ its image in $A_N$. Then under the identification $\phi^N$ defined in the proof of Theorem \ref{thm:periods},
$$
\alpha_{N^\perp}(\rho_{\mathcal{N}^\perp}(\gamma))(\phi^N(\overline{u})) = \phi^N(\alpha_N ( \rho_{\mathcal{N}}(\gamma))(\overline{u})).
$$
Since $\alpha_N$ is an injection and $\phi^N$ is an isomorphism, we see that, for an $(N,G_N)$-polarized family, all data about algebraic monodromy of $\mathcal{N}$ is captured by the monodromy of $\mathcal{N}^\perp$.
\end{remark}

This remark will be essential for the calculations that we will do in Section \ref{genconst}.

\section{Symplectic automorphisms in families} \label{section:siif}

In this section, we expand upon Proposition \ref{prop:sympaut} in the case where $\tau$ is a Nikulin involution. The main result is Theorem \ref{thm:involutions}, which will be used in Section \ref{genconst} to study lattice polarized families of K3 surfaces with Shioda-Inose structure, in an attempt to understand the relationship between such families and their associated families of abelian surfaces.

\subsection{Symplectic automorphisms and Nikulin involutions}\label{sect:aut}
\label{section:involutions}

We begin with some background on symplectic automorphisms of K3 surfaces. Let $X$ be a K3 surface and let $\omega$ be a non-vanishing holomorphic 2-form on $X$. For any group $\Sigma$ of symplectic automorphisms of $X$, there are two lattices in $H^2(X,\mathbb{Z})$ which may be canonically associated to $\Sigma$. The first is the fixed lattice $H^2(X,\mathbb{Z})^\Sigma$. To derive the second, note that, by assumption, $\Sigma$ fixes $\omega$ and hence, since $\Sigma$ acts as Hodge isometries on $H^2(X,\mathbb{Z})$, we see that $\Sigma$ must preserve the transcendental Hodge structure on $X$. This implies that $\T(X) \subseteq H^2(X,\mathbb{Z})^\Sigma$. So we may define a second lattice
\begin{equation*}\label{eq:SG}
S_{\Sigma,X} := (H^2(X,\mathbb{Z})^\Sigma)^\perp.
\end{equation*}

When the K3 surface $X$ is understood, we will abbreviate this notation to simply $S_\Sigma$. This is appropriate because Nikulin \cite[Theorem 4.7]{fagkk3s} proves that, as an abstract lattice, $S_\Sigma$ depends only upon $\Sigma$.  It follows from the fact that $\T(X)$ is fixed by $\Sigma$ that $S_\Sigma$ is contained in $\NS(X)$. In \cite[Lemma 4.2]{fagkk3s} it is also shown that $S_\Sigma$ is a negative definite lattice and contains no elements of square $(-2)$. 

In \cite[Proposition 7.1]{fagkk3s}, Nikulin determines the lattice $S_\Sigma$ for any abelian group of symplectic automorphisms $\Sigma$. Therefore, since any group contains at least one abelian subgroup, if $X$ admits any nontrivial group $\Sigma$ of symplectic automorphisms, then $S_\Sigma$ contains one of the lattices in \cite[Proposition 7.1]{fagkk3s}. The smallest lattice listed therein is $S_{\mathbb{Z}/2\mathbb{Z}}$, which has rank $8$.

\medskip

In general, symplectic automorphisms have fixed point sets of dimension $0$. The local behaviour of $\Sigma$ about the fixed points determines a quotient singularity in $X/\Sigma$. It is easy to see from the classification of minimal surfaces that the minimal resolution $Y:= \widetilde{X/\Sigma}$ of $X/\Sigma$ is again a K3 surface:  $\sigma^*\omega = \omega$ implies that $\omega$ descends to a non-vanishing holomorphic $2$-form on the quotient surface and the resulting quotient singularities are crepant.

There is a diagram of surfaces
\[\xymatrix{
& \tilde{X} \ar[ld]_{c} \ar[rd]^{q} &  \\
X \ar[rd] & & Y \ar[ld]\\
& X/\Sigma &\\
}\]
where $\tilde{X}$ is the minimal blow up of $X$ on which $\Sigma$ acts equivariantly with the map $c$ and whose quotient $\widetilde{X}/\Sigma$ is $Y$. 

In $\NS(Y)$ there is a lattice $K$ spanned by exceptional classes. The minimal primitive sublattice of $\NS(Y)$ containing $K$ will be called $K_0$. Nikulin \cite[Propositions 7.1 and 10.1]{fagkk3s} shows that $K_0$ and $S_\Sigma$ have the same rank but are, of course, not isomorphic. The map
$$
\theta := q^*c_*\colon K_0^\perp \rightarrow H^2(X,\mathbb{Z})^\Sigma
$$
is an isomorphism over $\mathbb{Q}$ and satisfies
$$
\langle \theta(u),\theta(v)\rangle = |\Sigma|\langle u,v\rangle
$$
for any $u,v \in K_0^\perp$. Therefore there is a linear transformation $g$ over $\mathbb{Q}(\sqrt{|\Sigma|})$ which relates the lattices $H^2(X,\mathbb{Z})^\Sigma$ and $K_0^\perp$; a more precise description of this relationship is given in \cite[Theorem 2.1]{sapgk3s}. 

Since the group $\Sigma$ acts symplectically, for a class $\omega$ spanning $H^{2,0}(Y)$ we have that $\theta(\omega)$ is in $H^{2,0}(X)$, so we see that $\langle \theta(u),\theta(\omega) \rangle = 0$ if and only if $\langle u,\omega\rangle =0$. Thus $\theta(\NS(Y)\cap K_0^\perp) =\NS(X)\cap H^2(X,\mathbb{Z})^\Sigma$. In other words, $\theta(\T(Y)) = \T(X)$.

\subsection{Symplectic quotients and Hodge bundles} \label{sect:sympquot}

If $\mathcal{X}$ is a family of K3 surfaces for which a group of symplectic automorphisms on the fibres extends to a group of automorphisms on the total space, then base-change allows us to relativize the constructions in Section \ref{sect:aut}.

We obtain sheaves of local systems $(R^2\pi_*\mathbb{Z})^\Sigma$ and $\mathcal{S}_\Sigma$ which agree fibrewise with $H^2(X_p,\mathbb{Z})^\Sigma$, and $S_{\Sigma,X_p}$. The Hodge filtration on $R^2\pi_*\mathbb{Z} \otimes \mathcal{O}_U$ restricted to these sub-sheaves produces integral weight $2$ variations of Hodge structure on $U$. 

We wish to compare the variation of Hodge structure on $(R^2\pi^\mathcal{X}_*\mathbb{Z})^\Sigma$ and the variation of Hodge structure on the subsystem of $R^2\pi^\mathcal{Y}_*\mathbb{Z}$ orthogonal to the lattice spanned by exceptional curves in each fibre. Since we deal only with smooth fibrations, the following statements are equivalent to their counterparts for individual K3 surfaces.

\begin{proposition}
\label{prop:hodge}
Let $\mathcal{X} \to U$ be a family of K3 surfaces on which a group $\Sigma$ of symplectic automorphisms acts fibrewise and extends to automorphisms of $\pi^\mathcal{X}\colon \mathcal{X}\rightarrow U$. Let $\pi^\mathcal{Y}\colon \mathcal{Y} \rightarrow U$ be the resolved quotient threefold. Then 
\begin{enumerate}[\textup{(}1\textup{)}]
\item The Hodge bundles $F^2(R^2\pi_*^\mathcal{X} \mathbb{Z}\otimes \mathcal{O}_U)$ and $F^2(R^2\pi_*^\mathcal{Y} \mathbb{Z}\otimes \mathcal{O}_U)$ are isomorphic as complex line bundles on $U$.

\item If we extend scalars to $\mathbb{Q}(\sqrt{|\Sigma|})$, the induced VHS on $(R^2\pi^\mathcal{X}_*\mathbb{Z})^\Sigma$ is isomorphic to a sub-VHS of $R^2\pi_*^\mathcal{Y} \mathbb{Z}$.
\item The transcendental integral variations of Hodge structure $\mathcal{T}(\mathcal{X})$ and $\mathcal{T}(\mathcal{Y})$ are isomorphic over $\mathbb{Q}(\sqrt{|\Sigma|})$.
\end{enumerate}
\end{proposition}
\begin{proof}

These are relative versions of the discussion in Section \ref{section:involutions}. We use the fact that statements about the local systems $R^2\pi_*^\mathcal{X}\mathbb{Z}$ and $R^2\pi^\mathcal{Y}_*\mathbb{Z}$ reduce to statements on each fibre. The same is true for statements about the Hodge filtrations on $R^2\pi_*^\mathcal{X}\mathbb{Z} \otimes \mathcal{O}_U$ and $R^2\pi_*^\mathcal{Y}\mathbb{Z}\otimes \mathcal{O}_U$. Therefore Proposition \ref{prop:hodge} reduces to the statements in Section \ref{section:involutions}.
\end{proof}

In particular, we can recover from Proposition \ref{prop:hodge}(3) a result of Smith \cite[Theorem 2.12]{pfdefk3s}, that the holomorphic Picard-Fuchs equation of $\mathcal{X}$ agrees with the Picard-Fuchs equation of $\mathcal{Y}$, since Picard-Fuchs equations depend only upon the underlying complex VHS.

A corollary to this is that the transcendental monodromy of $\mathcal{Y}$ can be calculated quite easily from the transcendental monodromy of $\mathcal{X}$. If we let $g$ be the $\mathbb{Q}(\sqrt{|\Sigma|})$-linear map relating the lattices $H^2(X_p,\mathbb{Z})^\Sigma$ and $K_0^{\perp}$
$$
g\colon H^2(X_p,\mathbb{Z})^\Sigma \rightarrow K_0^{\perp}
$$
for a given fibre $X_p$, then 
\begin{equation} \label{eq:quot}
\rho_{H^2(X_p,\mathbb{Z})^\Sigma}(w) = g^{-1} \rho_{K_0^{\perp}}(g \cdot w).
\end{equation}
In particular, we have:

\begin{corollary} \label{cor:hodge}
Let $\mathcal{X}$ be an $N$-polarized family of K3 surfaces and suppose $N \cong \NS(X_p)$ for some fibre $X_p$. Assume that $\mathcal{X}$ admits a group of fibrewise symplectic automorphisms $\Sigma$ and let $\mathcal{Y}$ be the fibrewise resolution of the quotient $\mathcal{X}/\Sigma$. If $K_0^{\perp}$ is the sublattice generated by classes orthogonal to exceptional curves on $Y_p$, then the monodromy representation fixes $K_0^{\perp} \cap \NS(Y_p)$.
\end{corollary}
\begin{proof}
By construction, we have that $\NS(X_p)^\Sigma$ is fixed under monodromy. Therefore, the relation in Equation \eqref{eq:quot} implies that its image in $K_0^{\perp}$ under the $\mathbb{Q}(\sqrt{|\Sigma|})$ isometry $g$ is also fixed. Since $g$ sends the transcendental lattice of $X_p$ to the transcendental lattice $Y_p$, the image of $\NS(X_p)^\Sigma$ under $g$ is $K_0^{\perp} \cap \NS(Y_p)$. Thus $K_0^{\perp} \cap \NS(Y_p)$ is fixed by monodromy of the family $\mathcal{Y}$.
\end{proof}

\subsection{Nikulin involutions in families}

We will now tie our results together. We begin with a family $\mathcal{X}$ of K3 surfaces which admits a fibrewise Nikulin involution and is lattice polarized by a lattice $N$ which is isomorphic to the generic N\'eron-Severi lattice of the fibres of $\mathcal{X}$. Our goal is to understand how lattice polarization behaves under Nikulin involutions in families. We begin with some generalities on Nikulin involutions.

A Nikulin involution fixes precisely $8$ points on $X$. The resulting quotient $X/\beta$ has $8$ ordinary double points which are then resolved by blowing up to give a new K3 surface $Y$. We can also resolve these singularities indirectly by blowing up $X$ at the $8$ fixed points of $\beta$, calling the resulting exceptional divisors $\{E_i\}_{i=1}^8$. We see that the blown up K3 surface $\tilde{X}$ also admits an involution $\tilde{\beta}$ whose fixed locus is the exceptional divisor
$$
D = \sum_{i=1}^8 E_i.
$$
Let $F_i = q_* E_i$, where $q\colon \tilde{X}\to \tilde{X}/\tilde{\beta} \cong Y$ is the quotient map. The branch divisor in $Y$ is then the sum $f_*D = \sum_{i=1}^8 F_i$. Since there is a double cover ramified over $f_*D$, there must be some divisor 
$$
B = \dfrac{1}{2} f_*D.
$$
We call the lattice generated by $B$ and $\{F_i \}_{i=1}^8$ the Nikulin lattice, which we denote $K_{\mathrm{Nik}}$. 

According to \cite[Section 6]{fagkk3s}, $K_{\mathrm{Nik}}$ is a primitive sublattice of $\NS(\tilde{X}/\tilde{\beta})$ and, in the case where $\Sigma$ is a group of order $2$, the lattice $K_0$ discussed in Section \ref{section:involutions} is equal to $K_{\mathrm{Nik}}$. The following theorem is a technical tool, useful for calculations in Section \ref{genconst}.

\begin{theorem}
\label{thm:involutions}
Let $\mathcal{X} \to U$ be an $N$-polarized family of K3 surfaces and suppose $N \cong \NS(X_p)$ for some fibre $X_p$. Suppose further that $X_p$ admits a Nikulin involution $\beta$; by Corollary \ref{cor:sympaut} this extends to an involution on $\mathcal{X}$. Let $\mathcal{Y} \rightarrow U$ be the resolved quotient family of K3 surfaces and let $N'$ be the N\'eron-Severi lattice of a generic fibre of $\mathcal{Y}$. Then there is a subgroup $G$ of $\Aut(A_{N'})$ for which $\mathcal{Y}$ is an $(N',G)$-polarized family of K3 surfaces.
\end{theorem}
\begin{proof} To see that the resulting family $\mathcal{Y}$ is $(N',G)$-polarized for some $G$, it is enough to see that monodromy of $\mathcal{Y}$ cannot act trivially on $\Aut(A_{N'})$. 

First we note that monodromy of $\mathcal{Y}$ must fix $K_{\mathrm{Nik}}^{\perp} \cap \NS(Y_p)$ by Corollary \ref{cor:hodge}, where $K_{\mathrm{Nik}}$ denotes the Nikulin lattice. Thus the only non-trivial action of monodromy can be upon $K_{\mathrm{Nik}}$. 

Suppose for a contradiction that the image of $\rho_{\mathcal{NS}(\mathcal{Y})}$ contains a non-identity element $g$ that lies in the kernel of $\alpha_{N'}$. Recall from Theorem \ref{thm:pss} that such a $g$ must act on $\NS(Y_p)$ in the same way as a non-trivial symplectic automorphism $\tau$. Thus the orthogonal complement of the fixed lattice $\NS(Y_p)^{g}$ must have rank at least $8$. Since $K_{\mathrm{Nik}}$ has rank 8 and $K_{\mathrm{Nik}}^{\perp} \cap \NS(Y_p)$ is fixed under monodromy, the orthogonal complement of $\NS(Y_p)^{g}$ must be contained in $K_{\mathrm{Nik}}$. For reasons of rank this containment cannot be strict, so we must have equality. However, $K_{\mathrm{Nik}}$ is generated by elements of square $(-2)$, thus, by \cite[Lemma 4.2]{fagkk3s}, it cannot be the lattice $S_{\tau}$ of any automorphism $\tau$ of $X$. This is a contradiction. \end{proof}

Note that the proof given above does not extend to quotients by arbitrary symplectic automorphisms.

As a result of this theorem, Remark \ref{rmk:alg-trans dictionary} and Equation \eqref{eq:quot} we may calculate $G$.

\begin{corollary} \label{cor:Gfind}
If $g$ is the linear transformation which relates $\T(X_p)$ to $\T(Y_p)$ for some $p \in U$ and $\Gamma_\mathcal{X}$ \textup{(}resp. $\Gamma_{\calY}$\textup{)} is the image of the monodromy group of $\mathcal{T}(\mathcal{X})$ in $\mathrm{O}(\T(X_p))$ \textup{(}resp. $\mathcal{T}(\mathcal{Y})$ in $\mathrm{O}(\T(Y_p))$\textup{)}, then $\Gamma_\mathcal{Y} = g^{-1} \Gamma_\mathcal{X} g$ and the image $\alpha_{\T(\calY)}(\Gamma_\mathcal{Y})$ is the group $G$ such that $\mathcal{Y}$ is minimally $(N',G)$-polarized.
\end{corollary}

This allows us to control the algebraic monodromy of the family $\mathcal{Y}$ of K3 surfaces. In the following section, we concern ourselves with a geometric situation where it will be important to know exactly what our algebraic monodromy looks like. 

\section{Undoing the Kummer construction.} \label{genconst}

One of the major motivations for this work is the idea of undoing the Kummer construction globally in families. As we shall see, this has applications to the study of Calabi-Yau threefolds.

\subsection{The general case.}\label{gencaseundo}

Begin by assuming that $\mathcal{X}$ is a family of K3 surfaces which admit Shioda-Inose structure. Concretely, a \emph{Shioda-Inose structure} on a K3 surface $X$ is an embedding of the lattice $E_8 \oplus E_8$ into $\NS(X)$. By \cite[Section 6]{k3slpn}, a Shioda-Inose structure defines a canonical Nikulin involution $\beta$ and the minimal resolution of the quotient $X/\beta$ is a Kummer surface. Furthermore, if $X$ has transcendental lattice $\T(X)$, then the resolved quotient $Y = \widetilde{X/ \beta}$ has transcendental lattice $\T(Y) \cong \T(X)(2)$.

Assume that $\mathcal{X}$ is a lattice polarized family of Shioda-Inose K3 surfaces. Then by Corollary \ref{cor:sympaut}, the Nikulin involution extends to the entire family of K3 surfaces to produce a resolved quotient family $\mathcal{Y}$ of Kummer surfaces. 

We would like to find conditions under which one may undo the Kummer construction \emph{in families} starting from the polarized family $\mathcal{X}$ of K3 surfaces with Shioda-Inose structure. In other words, we would like to find conditions under which a family of abelian surfaces $\calA$ exists, such that application of the Kummer construction fibrewise to $\calA$ yields the family $\mathcal{Y}$ of Kummer surfaces associated to $\mathcal{X}$.

The following proposition provides an easy sufficient condition for undoing the Kummer construction on a family of Kummer surfaces.

\begin{proposition} \label{prop:untwistingkummer}
Beginning with a family of lattice polarized Shioda-Inose K3 surfaces $\mathcal{X}$ over $U$, the Kummer construction can be undone on the family of resolved quotient K3 surfaces $\mathcal{Y}$, if $\mathcal{Y}$ itself is lattice polarized.
\end{proposition}

In general, however, the family $\mathcal{Y}$ will not be lattice polarized; instead, by Theorem \ref{thm:involutions}, it will be $(N',G)$-polarized, for some lattice $N'$ and subgroup $G$ of $\Aut(A_{N'})$. To rectify this, we will have to proceed to a cover $f\colon U' \to U$ to remove the action of the group $G$, so that the Kummer construction can be undone on the pulled-back family $f^*\mathcal{Y}$. 

We begin by finding such a group $G$. We note, however, that in general $\calY$ will not be \emph{minimally} $(N',G)$-polarized for this choice of $G$.

\begin{proposition} \label{prop:generalcaseG}
Let $\mathcal{X} \to U$ be a family of $N$-polarized K3 surfaces with Shioda-Inose structure, where $N$ is isometric to the N\'eron-Severi lattice of a generic K3 fibre $X_p$. Then the associated family of Kummer surfaces $\mathcal{Y}$ is an $(N',G)$-polarized family of K3 surfaces, where $N'$ is the generic N\'eron-Severi lattice of fibres of $\mathcal{Y}$ and $G$ is the group
$$
\mathrm{O}(N^{\perp})^*/ \mathrm{O}(N^{\perp}(2))^*.
$$
Furthermore, if $\mathcal{X}$ has transcendental monodromy group $\Gamma_\mathcal{X} = \mathrm{O}(N^\perp)^*$, then $\calY$ is minimally $(N',G)$-polarized.
\end{proposition}
\begin{proof} By the results of Section \ref{sect:sympquot} there is a map
$$
g\colon \rho_{\mathcal{T}(\mathcal{X})} \rightarrow \rho_{\mathcal{T}(\mathcal{Y})}.
$$
 Let $X_p$ be a general fibre of $\calX$ and let $Y_p$ be the associated fibre of $\calY$. As $X_p$ has Shioda-Inose structure and $Y_p$ is the associated Kummer surface, the transformation $g$ induces the identity map on the level of orthogonal groups,
$$
\Id\colon \mathrm{O}(\T(X_p)) \rightarrow \mathrm{O}(\T(Y_p))
$$
since the lattice $\T(Y_p)$ is just $\T(X_p)$ scaled by $2$. 

Let $\Gamma_{\calX}$ (resp. $\Gamma_{\calY}$) denote the transcendental monodromy group of $\calX$ (resp. $\calY$). Then, by Corollary \ref{cor:Gfind}, $\Gamma_{\calY} = g^{-1}\Gamma_{\calX}g \cong \Gamma_{\calX}$ and $\calY$ is minimally $(N',\alpha_{\T(\calY)}(\Gamma_{\calY}))$-polarized. But $\Gamma_\mathcal{X} \subset \mathrm{O}(\T(X_p))^* \cong \mathrm{O}(N^{\perp})^*$ (by \cite[Proposition 3.3]{mslpk3s}) and $\alpha_{\T(\calY)}$ has kernel $\mathrm{O}(\T(X_p)(2))^* \cong \mathrm{O}(N^{\perp}(2))^*$, so $\alpha_{\T(\calY)}(\Gamma_{\calY}) \subset G$, where $G$ is as in the statement of the proposition, with equality if $\Gamma_\mathcal{X} = \mathrm{O}(N^\perp)^*$.\end{proof}

The group $G$ from this proposition will prove to be very useful in later sections.

\subsection{\texorpdfstring{$M$}{M}-polarized K3 surfaces.} \label{M-polarized}

We will be particularly interested in the case in which our family $\mathcal{X}$ is $M$-polarized, where $M$ denotes the lattice
\[ M := H \oplus E_8 \oplus E_8.\]
Such families admit canonically defined Shioda-Inose structures, so the discussion from Section \ref{gencaseundo} holds.

Our interest in such familes stems from the paper \cite{doranmorgan}, in which Doran and Morgan explicitly classify the possible integral variations of Hodge structure that can underlie a family of Calabi-Yau threefolds over $\Proj^1-\{0,1,\infty\}$ with $h^{2,1} = 1$. Their classification is given in \cite[Table 1]{doranmorgan}, which divides the possibilities into fourteen cases. Explicit examples, arising from toric geometry, of families of Calabi-Yau threefolds realising thirteen of these cases were known at the time of publication of \cite{doranmorgan} and are given in the rightmost column of \cite[Table 1]{doranmorgan}. A family of Calabi-Yau threefolds that realised the missing case (hereafter known as the \emph{14th case}) was constructed in \cite{14thcase}.

It turns out that many of these threefolds admit fibrations by $M$-polarized K3 surfaces. The ability to undo the Kummer construction globally on such threefolds therefore provides a new perspective on the geometry of the families in  \cite[Table 1]{doranmorgan}, which will be explored further in the remainder of this paper.

We begin this discussion with a brief digression into the geometry of $M$-polarized K3 surfaces, that we will need in the subsequent sections. In this section we will denote an $M$-polarized K3 surface by $(X,i)$, where $X$ is a K3 surface and $i$ is an embedding $i \colon M \hookrightarrow \mathrm{NS}(X)$.

Clingher, Doran, Lewis and Whitcher \cite{nfk3smmp} have shown that $M$-polarized K3 surfaces have a coarse moduli space given by the locus $d \neq 0$ in the weighted projective space $\WP(2,3,6)$ with weighted coordinates $(a,b,d)$. Thus, by normalizing $d = 1$, we may associate a pair of complex numbers $(a,b)$ to an $M$-polarized K3 surface $(X,i)$.

Let $\beta$ denote the Nikulin involution defined by the canonical Shioda-Inose structure on $(X,i)$. Then Clingher and Doran \cite[Theorem 3.13]{milpk3s} have shown that the resolved quotient $Y = \widetilde{X/\beta}$ is isomorphic to the Kummer surface $\mathrm{Kum}(A)$, where $A \cong E_1 \times E_2$ is an Abelian surface that splits as a product of elliptic curves. By \cite[Corollary 4.2]{milpk3s} the $j$-invariants of these elliptic curves are given by the roots of the equation
\[ j^2 - \sigma j + \pi = 0,\]
where $\sigma$ and $\pi$ are given in terms of the $(a,b)$ values associated to $(X,i)$ by $\sigma = a^3 - b^2+1$ and $\pi = a^3$. 

There is one final piece of structure on $(X,i)$ that we will need in our discussion. By \cite[Proposition 3.10]{milpk3s}, the K3 surface $X$ admits two uniquely defined elliptic fibrations $\Theta_{1,2}\colon X \to \Proj^1$, the \emph{standard} and \emph{alternate fibrations}. We will be mainly concerned with the alternate fibration $\Theta_2$. This fibration has two sections, one singular fibre of type $I_{12}^*$ and, if $a^3 \neq (b \pm 1)^2$, six singular fibres of type $I_1$ \cite[Proposition 4.6]{milpk3s}. Moreover, $\Theta_2$ is preserved by the Nikulin involution $\beta$, so induces a fibration $\Psi\colon Y \to \Proj^1$ on $X$. The two sections of $\Theta_2$ are identified to give a section of $\Psi$, and $\Psi$ has one singular fibre of type $I_{6}^*$ and, if $a^3 \neq (b \pm 1)^2$, six $I_2$'s \cite[Proposition 4.7]{milpk3s}.

\subsection{Undoing the Kummer construction for \texorpdfstring{$M$}{M}-polarized families}\label{undokummer}

We will use this background to outline a method by which we can undo the Kummer construction for a family obtained as a resolved quotient of an $M$-polarized family of K3 surfaces. An illustration of the use of this method to undo the Kummer construction in an explicit example may be found in \cite[Section 7.1]{14thcase}.

Let $N$ be a lattice that contains a sublattice isomorpic to $M$. Assume that $\calX$ is an $N$-polarized family of K3 surfaces over $U$ with generic N\'{e}ron-Severi lattice $N \cong NS(X_p)$, where $X_p$ is the fibre over a general point $p \in U$. Choose an embedding $M \hookrightarrow \mathrm{NS}(X_p)$; this extends uniquely to all other fibres of $\calX$ by parallel transport and thus exhibits $\calX$ as an $M$-polarized family of K3's.

This $M$-polarization induces a Shioda-Inose structure on the fibres of $\calX$, which defines a canonical Nikulin involution on these fibres that extends globally by Corollary \ref{cor:sympaut}. Define $\calY$ to be the variety obtained from $\calX$ by quotienting by this fibrewise Nikulin involution and resolving the resulting singularities. Then $\calY$ is fibred over $U$ by Kummer surfaces associated to products of elliptic curves. Let $Y_p \cong \mathrm{Kum}(E_1 \times E_2)$ denote the fibre of $\calY$ over the point $p \in U$, where $E_1$ and $E_2$ are elliptic curves.

The aim of this section is to find a cover $\calY'$ of $\calY$ upon which we can undo the Kummer construction. The results of Section \ref{gencaseundo} give a way to do this. Let $N' \cong \mathrm{NS}(Y_p)$ denote the generic N\'{e}ron-Severi lattice of $\calY$. Then Theorem \ref{thm:involutions} shows that there is a subgroup $G$ of $\Aut(A_{N'})$ for which $\calY$ is an $(N',G)$-polarized family of K3 surfaces. We will find a way to compute the action of monodromy around loops in $U$ on $N'$, which will allow us to find the group $G$ such that $\calY$ is a minimally $(N',G)$-polarized family, along with a cover $\calY'$ of $\calY$ that is an $N'$-polarized family of K3 surfaces. Then Proposition \ref{prop:untwistingkummer} shows that we can undo the Kummer construction on $\calY'$.

To simplify this problem we note that, by Corollary \ref{cor:hodge}, the only non-trivial action of monodromy on $N'$ can be on the Nikulin lattice $K_{\mathrm{Nik}}$ contained within it. This lattice is generated by the eight exceptional curves $F_i$ obtained by blowing up the fixed points of the Nikulin involution. Moreover, as $\beta$ extends to a global involution on $\calX$, the set $\{F_1,\ldots,F_8\}$ is preserved under monodromy (although the curves themselves may be permuted). Thus, we can compute the action of monodromy on $N'$ by studying its action on the curves $F_i$.

To find these curves, we begin by studying the configuration of divisors on a general fibre $Y_p$. Recall that $Y_p$ is isomorphic to $\mathrm{Kum}(E_1 \times E_2)$, where $E_1$ and $E_2$ are elliptic curves. There is a special configuration of twenty-four $(-2)$-curves on $\mathrm{Kum}(E_1 \times E_2)$ arising from the Kummer construction, that we shall now describe (here we note that we use the same notation as \cite[Definition 3.18]{milpk3s}, but with the roles of $G_i$ and $H_j$ reversed).

Let $\{x_0,x_1,x_2,x_3\}$ and $\{y_0,y_1,y_2,y_3\}$ denote the two sets of points of order two on $E_1$ and $E_2$ respectively. Denote by $G_i$ and $H_j$ ($0 \leq i,j \leq 3$) the $(-2)$-curves on $\mathrm{Kum}(E_1 \times E_2)$ obtained as the proper transforms of $E_1 \times \{y_i\}$ and $\{x_j\} \times E_2$ respectively. Let $E_{ij}$ be the exceptional $(-2)$-curve on $\mathrm{Kum}(E_1 \times E_2)$ associated to the point $(x_j,y_i)$ of $E_1 \times E_2$. This gives $24$ curves, which have the following intersection numbers:

\begin{align*}
G_i.H_j &= 0, \\
G_k.E_{ij} &= \delta_{ik}, \\
H_k.E_{ij} &= \delta_{jk}.
\end{align*}

\begin{defn} The configuration of twenty-four $(-2)$-curves 
\[ \{G_i,H_j,E_{ij} \mid 0 \leq i,j \leq 3 \}\]
is called a \emph{double Kummer pencil} on $\mathrm{Kum}(E_1 \times E_2)$.
\end{defn}

\begin{remark} Note that there may be many distinct double Kummer pencils on $\mathrm{Kum}(E_1 \times E_2)$. However, if $E_1$ and $E_2$ are non-isogenous, Oguiso \cite[Lemma 1]{ojfkspniec} shows that any two double Kummer pencils are related by a symplectic automorphism on $\mathrm{Kum}(E_1 \times E_2)$.
\end{remark}

Clingher and Doran \cite[Section 3.4]{milpk3s} identify such a pencil on the resolved quotient of an $M$-polarised K3 surface. We will study this pencil on a fibre of $\calY$ and, by studying the action of monodromy on it, derive the action of monodromy on the curves $F_i$. 

By the discussion in Section \ref{M-polarized}, the $M$-polarization structure on $X_p$ defines an elliptic fibration $\Theta_2$ on it, which is compatible with the Nikulin involution. Furthermore, as $\calX$ is an $M$-polarized family, this elliptic fibration extends to all fibres of $\calX$ and is compatible with the fibrewise Nikulin involution. Therefore $\Theta_2$ induces an elliptic fibration $\Psi$ on $Y_p$ which extends uniquely to all fibres of $\calY$, so $\Psi$ must be preserved under the action of monodromy around loops in $U$.

Using the same notation as in \cite[Diagram (26)]{milpk3s}, we may label some of the $(-2)$-curves in the fibration $\Psi$ as follows: 
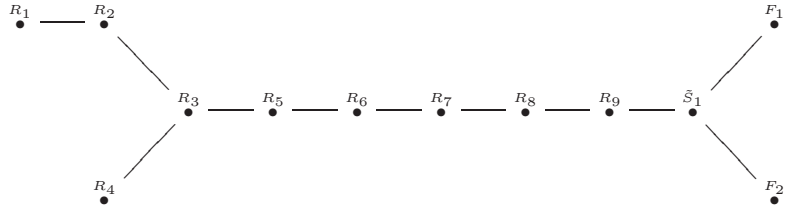
\begin{equation*}
\def\objectstyle{\scriptstyle}
\def\labelstyle{\scriptstyle}
\xymatrix @-0.6pc  
{
\stackrel{R_1}{\bullet} \ar @{-} [r] & \stackrel{R_2}{\bullet} \ar @{-} [dr] & & & &  & &  & & \stackrel{F_{1}}{\bullet} \ar @{-} [dl] 
 \\
& & \stackrel{R_3}{\bullet} \ar @{-} [r] \ar @{-} [dl] &
\stackrel{R_5}{\bullet} \ar @{-} [r] &
\stackrel{R_6}{\bullet} \ar @{-} [r] &
\stackrel{R_7}{\bullet} \ar @{-} [r] &
\stackrel{R_8}{\bullet} \ar @{-} [r] &
\stackrel{R_9}{\bullet} \ar @{-} [r] &
\stackrel{\tilde{S}_1}{\bullet}  \ar @{-} [dr]   \\
&  \stackrel{R_{4}}{\bullet} & & & & & & & & \stackrel{F_{2}}{\bullet} 
}
\end{equation*}
Here $R_1$ is the section of $\Psi$ given uniquely as the image of the two sections of $\Theta_2$ and the remaining curves form the $I_6^*$ fibre. Note that the $R_i$ and $\tilde{S}_1$ are uniquely determined by the structure of $\Psi$, so must be invariant under the action of monodromy around loops in $U$. By the discussion in \cite[Section 3.5]{milpk3s} the curves $F_1$ and $F_2$ are two of the eight exceptional curves that we seek, but are determined only up to permutation.

By the discussion in \cite[Section 4.6]{milpk3s}, we may identify these curves with $(-2)$-curves in a double Kummer pencil as follows: $R_1 = G_2$, $R_2 = E_{20}$, $R_3 = H_0$, $R_4 = E_{30}$, $R_5 = E_{10}$, $R_6 = G_1$, $R_7 = E_{11}$, $R_8 = H_1$, $R_9 = E_{01}$, $\tilde{S}_1 = G_0$, $F_1 = E_{02}$ and $F_2 = E_{03}$. This gives:

\begin{lemma} \label{10fixed} In the double Kummer pencil on $Y_p$ defined above, the action of monodromy around loops in $U$ must fix the $10$ curves $G_0$, $G_1$, $G_2$, $H_0$, $H_1$, $E_{01}$, $E_{10}$, $E_{11}$, $E_{20}$, $E_{30}$. \end{lemma}

We can improve on this result, but in order to do so we will need to make an assumption:

\begin{assumption}\label{I2ass} The fibration $\Psi$ on $Y_p$ has six singular fibres of type $I_2$.\end{assumption}

\begin{remark}  \label{rem:assumption} Recall from the discussion in Section \ref{M-polarized} that this assumption is equivalent to the assumption that the $(a,b)$-parameters of the $M$-polarized fibre $X_p$ satisfy $a^3 \neq (b \pm 1)^2$.\end{remark}

Using this, we may now identify all eight of the curves $F_i$. From the discussion above, we already know $F_1 = E_{02}$ and $F_2 = E_{03}$. \cite[Section 3.5]{milpk3s} shows that, under Assumption \ref{I2ass}, the remaining six $F_i$ are the components of the six $I_2$ fibres in $\Psi$ that are disjoint from the section $R_1 = G_2$.

Kuwata and Shioda \cite[Section 5.2]{epdeefks} explicitly identify these six $I_2$ fibres in the double Kummer pencil on $Y_p$. We see that:
\begin{itemize}
\item the section $G_3$ of $\Psi$ is the unique section that intersects all six of $F_3,\ldots,F_8$,
\item the section $H_2$ of $\Psi$ intersects $F_1$ and precisely three of $F_3,\ldots,F_8$ (say $F_3$, $F_4$, $F_5$), and
\item the section $H_3$ of $\Psi$ intersects $F_2$ and the other three $F_3,\ldots,F_8$ (say $F_6$, $F_7$, $F_8$).
\end{itemize}

Combining this with Lemma \ref{10fixed} and the fact that the structure of $\Psi$ is preserved under monodromy, we obtain

\begin{proposition} \label{11fixed} In addition to fixing the ten curves from Lemma \ref{10fixed}, the action of monodromy around a loop in $U$ must also fix $G_3$ and either
\begin{enumerate}[\textup{(}1\textup{)}]
\item fix both $F_1 = E_{02}$ and $F_2 = E_{03}$, in which case $H_2$ and $H_3$ are also fixed and the sets $\{F_3,F_4,F_5\}$ and $\{F_6,F_7,F_8\}$ are both preserved, or
\item interchange $F_1 = E_{02}$ and $F_2 = E_{03}$, in which case $H_2$ and $H_3$ are also swapped and the sets $\{F_3,F_4,F_5\}$ and $\{F_6,F_7,F_8\}$ are interchanged.
\end{enumerate}
\end{proposition}

Whether the action of monodromy around a given loop fixes or exchanges $F_1 = E_{02}$ and $F_2 = E_{03}$ may be calculated explicitly. Recall that the curves $\{F_3,\ldots,F_8\}$ appear as components of the $I_2$ fibres in the alternate fibration on $Y_p$.  Let $x$ be an affine parameter on the base $\Proj^1_x$ of the alternate fibration on $Y_p$, chosen so that the $I_6^*$-fibre occurs at $x = \infty$. Then the locations of the $I_2$ fibres is given explicitly by \cite[Proposition 4.7]{milpk3s}: they lie at the roots of the polynomials $(P(x) \pm 1)$, where
\begin{equation}\label{Pequation} P(x) := 4x^3 - 3ax - b, \end{equation}
for $a$ and $b$ the $(a,b)$-parameters associated to the $M$-polarized K3 surface $X_p$.

Without loss of generality, we may say that $\{F_3,F_4,F_5\}$ appear in the $I_2$ fibres occurring at roots of $(P(x)-1)$ and $\{F_6,F_7,F_8\}$ appear in the $I_2$ fibres occurring at roots of $(P(x)+1)$. We thus have:

\begin{corollary} \label{paircor} Case \textup{(}1\textup{)} \textup{(}resp. \textup{(}2\textup{)}\textup{)} of Proposition \ref{11fixed} holds for monodromy around a given loop if and only if that monodromy preserves the set of roots of $(P(x)+1)$ \textup{(}resp. switches the sets of roots of the polynomials $(P(x)+1)$ and $(P(x)-1)$\textup{)}.\end{corollary} 

If case (2) of Proposition \ref{11fixed} holds for some loop in $U$, we note that the Nikulin lattice is \emph{not} fixed under monodromy around that loop. This presents an obstruction to $\calY$ admitting an $N'$-polarization. To resolve this we may pull-back $\calY$ to a double cover of $U$, after which case (1) of the lemma will hold around all loops and the curves $F_1 = E_{02}$, $F_2 = E_{03}$, $H_2$ and $H_3$ will all be fixed under monodromy.

Given this, we may safely assume that case (1) holds around all loops in $U$, so $F_1$ and $F_2$ are fixed under monodromy and the sets $\{F_3,F_4,F_5\}$ and $\{F_6,F_7,F_8\}$ are both preserved. All that remains is to find whether monodromy acts to permute $F_3,\ldots,F_8$ within these sets.

\begin{proposition} \label{Fmonodromy} Assume that the action of monodromy around all loops in $U$ fixes both $F_1$ and $F_2$ \textup{(}i.e. case \textup{(}1\textup{)} of Proposition \ref{11fixed} holds around all loops in $U$\textup{)}. Then the action of monodromy around a loop in $U$ permutes $\{F_3,F_4,F_5\}$ \textup{(}resp. $\{F_6,F_7,F_8\}$\textup{)} if and only if it permutes the roots of $(P(x)-1)$ \textup{(}resp. $(P(x)+1)$\textup{)}.
\end{proposition}
\begin{proof} As $\{F_3,F_4,F_5\}$ appear in the $I_2$ fibres occurring at roots of $(P(x)-1)$ and $\{F_6,F_7,F_8\}$ appear in the $I_2$ fibres occurring at roots of $(P(x)+1)$, they are permuted if and only if the corresponding roots of $(P(x)-1)$ and $(P(x)+1)$ are permuted.\end{proof}

Monodromy around a loop thus acts on $\{F_3,F_4,F_5\}$ and $\{F_6,F_7,F_8\}$ as a permutation in $S_3 \times S_3$. Taken together, the permutations corresponding to monodromy around all loops generate a subgroup $H$ of $S_3 \times S_3$. 

Therefore, in order to obtain a $N'$-polarization on $\calY$, we need to pull everything back to a $|H|$-fold cover $f\colon V \to U$. This cover is constructed as follows: the $|H|$ preimages of the point $p \in U$ are labelled by permutations in $H$ and, if $\gamma$ is a loop in $U$, monodromy around $f^{-1}(\gamma)$ acts on these labels as composition with the corresponding permutation. This action extends to an action of $H$ on the whole of $V$. In fact, we have:

\begin{theorem} \label{thm:dekummer} Let $f\colon V \to U$ be the cover constructed above and let $\calY' \to V$ denote the pull-back of $\calY \to U$. Then $\calY'$ is a $N'$-polarized family, where $N'$ is the generic N\'{e}ron-Severi lattice of $\calY$, so we can undo the Kummer construction on $\calY'$. Furthermore, the deck transformation group of $f$ is a subgroup $G$ of $S_6$ given by:
\begin{itemize}
\item If case \textup{(}1\textup{)} of Proposition \ref{11fixed} holds around all loops in $U$, then $G = H$.
\item If case \textup{(}2\textup{)} of Proposition \ref{11fixed} holds around some loop in $U$, then there is an exact sequence $1 \to H \to G \to C_2 \to 1$.
\end{itemize}
\end{theorem}

\begin{remark} We note that in the second case there does not seem to be any reason to believe that $G \cong H \rtimes C_2$ in general. Whilst we do not know of any explicit examples where this fails, it does not seem to be inconsistent with the theory as presented. \end{remark}

\begin{proof} Let $Y'_p$ denote one of the preimages of $Y_p$ under the pull-back. Then the argument above shows that each of the eight curves $F_i$ extends uniquely to all smooth fibres of $\calY'$. Thus the Nikulin lattice $K_{\mathrm{Nik}}$ is preserved under monodromy and so, by Corollary \ref{cor:hodge}, $N'$ is also. Therefore $\calY'$ is a $N'$-polarized family and, by Proposition \ref{prop:untwistingkummer}, we may undo the Kummer construction on $\calY'$.

It just remains to verify the statements about the group $G$. Note that $G$ can be seen as a subgroup of $S_6$, given by permutations of the divisors $\{F_3,\ldots,F_8\}$, and that $H$ is the subgroup of $G$ given by those permutations that preserve the sets $\{F_3,F_4,F_5\}$ and $\{F_6,F_7,F_8\}$. If case (1) of Proposition \ref{11fixed} holds around all loops in $U$, then all permutations in $G$ preserve the sets $\{F_3,F_4,F_5\}$ and $\{F_6,F_7,F_8\}$, so $G = H$. If case (2) of Proposition \ref{11fixed} holds around some loop in $U$ then $H$ has index $2$ in $G$, so it must be a normal subgroup with quotient $G/H \cong C_2$.\end{proof}

\begin{corollary} \label{cor:dekummer} $\calY$ is a minimally $(N',G)$-polarized family of K3 surfaces, where $G$ is the group from Theorem \ref{thm:dekummer}.\end{corollary}

\begin{proof} We just need to show that $G$ is minimal. Note that $G$ was constructed explicitly as the permutation group of the divisors $\{F_1,\ldots,F_8\}$ under monodromy. Furthermore, it is clear from the construction that any permutation in $G$ is induced by monodromy around some loop in $U$. So $\alpha_{N'}$ is surjective and $G$ is minimal.\end{proof}

\begin{remark} \label{rem:dekummer} As the group $G$ from Theorem \ref{thm:dekummer} is minimal, it will be a subgroup of the group $\mathrm{O}(N^{\perp})^*/ \mathrm{O}(N^{\perp}(2))^*$ from Proposition \ref{prop:generalcaseG}.
\end{remark}

\subsection{The generically \texorpdfstring{$M$}{M}-polarized case.}

Suppose now that we are in the case where a general fibre $X_p$ of $\calX$ has $\NS(X_p) \cong M$. In this case we have the following version of Proposition \ref{prop:generalcaseG}.

\begin{proposition} \label{MG} Suppose that $\calX$ is an $M$-polarized family of K3 surfaces with general fibre $X_p$ satisfying $\NS(X_p) \cong M$. Then the resolved quotient $\calY \cong \widetilde{\calX/\beta}$ of $\calX$ by the fibrewise Nikulin involution is a \textup{(}not necessarily minimally\textup{)} $(N',G)$-polarized family of K3 surfaces, where $G \cong (S_3 \times S_3) \rtimes C_2$.\end{proposition}
\begin{proof} Recall that $M^{\perp}$ is isomorphic to $H^2$. The proposition will follow from Proposition \ref{prop:generalcaseG} if we can show that 
\[\mathrm{O}(H^2)^*/ \mathrm{O}(H^2(2))^* \cong  (S_3 \times S_3) \rtimes C_2.\]

This quotient is just $\Aut(A_{H^2(2)})$. To see this, note that $\mathrm{O}(H^2)^*$ is isomorphic to $\mathrm{O}(H^2)$, since $A_{H^2}$ is the trivial group, and $\mathrm{O}(H^2)$ is isomorphic to $\mathrm{O}(H^2(2))$, hence
\[\mathrm{O}(H^2(2))/ \mathrm{O}(H^2(2))^* \cong  \mathrm{O}(H^2)^*/ \mathrm{O}(H^2(2))^*.\]
By a standard lattice theoretic fact (see, for example, \cite[Theorem 3.6.3]{isbfa}), $\mathrm{O}(H^2(2))$ maps surjectively onto $\Aut(A_{H^2(2)})$. So the group $\mathrm{O}(H^2)^*/ \mathrm{O}(H^2(2))^*$ is isomorphic to $\Aut(A_{H^2(2)})$. According to \cite[Lemma 3.5]{agksaptec} this group is isomorphic to $(S_3 \times S_3) \rtimes C_2$.
\end{proof}

\begin{remark} The results of Section \ref{undokummer} give an immediate interpretation for this group: the two $S_3$ factors correspond to permutations of the two sets of divisors $\{F_3,F_4,F_5\}$ and $\{F_6,F_7,F_8\}$, whilst the $C_2$ corresponds to the action which interchanges these two sets (and also swaps $F_1$ and $F_2$).
\end{remark}

\begin{ex} \label{ex:14thcase} In \cite{14thcase}, the family of threefolds $Y_1$ that realise the 14th case variation of Hodge structure admit torically induced fibrations by $M$-polarized K3 surfaces with general fibre $X_p$ satisfying $\NS(X_p) \cong M$. In \cite[Section 7.1]{14thcase} we apply the results of the previous section to undo the Kummer construction for the resolved quotient $W \cong \widetilde{Y_1/\beta}$ of $Y_1$ by the fibrewise Nikulin involution. It is an easy consequence of those calculations that $W$  is \emph{minimally} $(N',G)$-polarized, for $G \cong (S_3 \times S_3) \rtimes C_2$.\end{ex}

It turns out, however, that the 14th case is the only case from \cite[Table 1]{doranmorgan} that admits a torically induced $M$-polarized fibration with general fibre $X_p$ satisfying $\NS(X_p) \cong M$. In most other cases (see Theorem \ref{thm:Mnfamilies}) the N\'{e}ron-Severi lattice of the general fibre is a lattice enhancement of $M$ to a lattice
\[M_n := M \oplus \langle -2n \rangle,\]
with $1 \leq n \leq 4$. In particular, note that $M_n$-polarized K3 surfaces are also $M$-polarized, so the analysis of this section still holds. We will examine this case in the next section.

\section{Threefolds fibred by \texorpdfstring{$M_n$}{Mn}-polarized K3 surfaces.}\label{sect:Mnthreefolds}

In this section we will specialize the analysis of Section \ref{genconst} to the case where we have a family $\calX$ of $M_n$-polarized K3 surfaces. We will then apply this theory to study $M_n$-polarized families of K3 surfaces arising from threefolds in the Doran-Morgan classification \cite[Table 1]{doranmorgan}.

\subsection{The groups \texorpdfstring{$G$}{G}.}

We begin with the analogue of Proposition  \ref{prop:generalcaseG} in the $M_n$-polarized case.

\begin{proposition} \label{MnG} Suppose that $\calX$ is an $M_n$-polarized family of K3 surfaces with general fibre $X_p$ satisfying $\NS(X_p) \cong M_n$. Then the resolved quotient $\calY \cong \widetilde{\calX/\beta}$ of $\calX$ by the fibrewise Nikulin involution is a \textup{(}not necessarily minimal\textup{)} $(N',G)$-polarized family of K3 surfaces, where $N'$ is the generic N\'{e}ron-Severi latice of $\calY$ and
\begin{itemize}
\item if $n = 1$ then $G = S_3 \times C_2$,
\item  if $n =2$ then $G = D_8$, the dihedral group of order $8$,
\item if $n = 3$ then $G = D_{12}$, and 
\item if $n = 4$ then $G = D_8$.
\end{itemize}
\end{proposition}

\begin{proof} This will follow from Proposition \ref{prop:generalcaseG} if we can show that 
\[\mathrm{O}(M_n^{\perp})^*/ \mathrm{O}(M_n^{\perp}(2))^* \cong  G,\]
where $G$ is as in each of the four cases in the statement of the proposition. We proceed by obtaining generators for $\mathrm{O}(M_n^{\perp})^* \cong \mathrm{O}(H \oplus \langle 2n\rangle)^*$ and then determining their actions on $A_{H(2) \oplus \langle 4n\rangle}$ to compute the group $G$.

In the case $n = 1$, the generators of $\mathrm{O}(H\oplus\langle 2\rangle)^*$ are
$$
g_1 = \left(\begin{array}{rrr}
0 & -1 & 0 \\
-1 & 0 & 0 \\
0 & 0 & 1
\end{array}\right),\quad
g_2 = \left(\begin{array}{rrr}
1 & 0 & 0 \\
1 & 1 & 2 \\
1 & 0 & 1
\end{array}\right) ,\quad
g_3 = \left(\begin{array}{rrr}
-1 & 0 & 0 \\
0 & -1 & 0 \\
0 & 0 & -1
\end{array}\right)
$$
whose induced actions on $A_{H(2) \oplus \langle 4\rangle}$ have orders $2$, $3$ and $2$ respectively. One may check that $g_1 g_2 g_1 = g_2^2$, and hence $g_1$ and $g_2$ generate a copy of $S_3$. It is clear that $g_3$ commutes with $g_1$ and $g_2$, so the subgroup of $\Aut(A_{H(2) \oplus \langle 4\rangle})$ generated by $g_1,g_2$ and $g_3$ is isomorphic to $S_3 \times C_2$.

In the case $n = 2$, the group $\mathrm{O}(H\oplus \langle 4 \rangle)^*$ has a non-minimal set of generators
$$
g_1 = \left(\begin{array}{rrr}
1 & 0 & 0 \\
2 & 1 & 4 \\
1 & 0 & 1
\end{array}\right), \quad
g_2 = \left(\begin{array}{rrr}
1 & 2 & 4 \\
2 & 1 & 4 \\
-1 & -1 & -3
\end{array}\right), \quad
g_3 = \left(\begin{array}{rrr}
0 & 1 & 0 \\
1 & 0 & 0 \\
0 & 0 & 1
\end{array}\right).
$$
Let the automorphism induced on $A_{H(2) \oplus \langle 8\rangle}$ by $g_i$ be denoted $h_i$. Then $h_1^2 = h_2^2 = h^2_3 = \Id$. We check $h_1 h_3$ has order 4 and it is easy to see that 
$$
h_1 (h_1 h_3) h_1 = h_3 h_1 = (h_1 h_3)^{-1}.
$$
Therefore, $h_1$ and $h_1h_3$ generate a copy of $D_8$. Finally, one checks that $(h_1 h_3) h_1 = h_2$, so the group of automorphisms $\langle h_1,h_2,h_3\rangle$ is isomorphic to $D_8$.

In the case when $n=3$, we may calculate generators of $\mathrm{O}( H \oplus \langle 6\rangle)^*$ to find
$$
g_1 = \left(\begin{array}{rrr}
1 & 0 & 0 \\
3 & 1 & 6 \\
1 & 0 & 1
\end{array}\right), \quad
g_2 = \left(\begin{array}{rrr}
1 & 3 & 6 \\
3 & 4 & 12 \\
-1 & -2 & -5
\end{array}\right), \quad
g_3 = \left(\begin{array}{rrr}
0 & 1 & 0 \\
1 & 0 & 0 \\
0 & 0 & 1
\end{array}\right).
$$
As before, let the corresponding automorphisms of $H(2) \oplus \langle 12\rangle$ be called $h_1,h_2$ and $h_3$. We calculate that 
$$
h_1^2= h_2^3 = h_3^2 = (h_1h_3)^6 = \Id.
$$
Furthermore, $(h_1h_3)^2 = h_2$ and 
$$
h_1 (h_1 h_3) h_1 = h_3 h_1 = (h_1 h_3)^{-1}.
$$
Therefore, the group $\langle h_1,h_2,h_3\rangle$ is isomorphic to $D_{12}$.

In the case when $n=4$, we may calculate generators of $\mathrm{O}( H \oplus \langle 8\rangle)^*$ to obtain
$$
g_1 = \left(\begin{array}{rrr}
1 & 0 & 0 \\
4 & 1 & 8 \\
1 & 0 & 1
\end{array}\right), \quad
g_2 = \left(\begin{array}{rrr}
9 & 4 & 24 \\
4 & 1 & 8 \\
-3 & -1 & -7
\end{array}\right), \quad
g_3 = \left(\begin{array}{rrr}
0 & 1 & 0 \\
1 & 0 & 0 \\
0 & 0 & 1
\end{array}\right).
$$
Once again, let the corresponding automorphisms of $H(2) \oplus \langle 16\rangle$ be called $h_1,h_2$ and $h_3$. We calculate that 
$$
h_1^2 = h_2^2 = h_3^2 = \Id
$$
We check that $h_1h_2 = h_2h_1$ and $h_3h_2 = h_2h_3$. Once again, we also have $(h_1h_3)^2 = h_2$ and
$$
h_1 (h_1 h_3) h_1 = h_3 h_1 = (h_1 h_3)^{-1}.
$$
Therefore the group $\langle h_1,h_2,h_3\rangle$ is isomorphic to $D_8$.
\end{proof}

\subsection{Some special families}\label{section:specialfamilies}

There are some special families of $M_n$-polarized K3 surfaces that we can use to vastly reduce the amount of work that we have to do to undo the Kummer construction for the $M_n$-polarized cases from \cite[Table 1]{doranmorgan}.

We begin by noting that the moduli space $\calM_{M_n}$ of $M_n$-polarized K3 surfaces is a $1$-dimensional modular curve \cite[Theorem 7.1]{mslpk3s}. Denote by $U_{M_n}$ the open subset of $\calM_{M_n}$ obtained by removing the orbifold points.  

\begin{defn} \label{Xndefn} $\calX_n \to U_{M_n}$ will denote an $M_n$-polarized family of  K3 surfaces over $U_{M_n}$, with period map $U_{M_n} \to \calM_{M_n}$ given by the inclusion and transcendental monodromy group $\Gamma_{\calX_n} = \mathrm{O}(M_n^\perp)^*$.
\end{defn}

\begin{remark} Examples of such families for any $n$ are given by the restriction of the special $M$-polarized family from \cite[Theorem 3.1]{nfk3smmp} to the $M_n$-polarized loci calculated in \cite[Section 3.2]{nfk3smmp}. For $n \leq 4$, we will explicitly construct examples of such families in Sections \ref{sect:14caseapp} and \ref{sect:M1}.\end{remark}

Let $\mathcal{Y}_n \to U_{M_n}$ be the family of Kummer surfaces associated to $\calX_n \to U_{M_n}$ and let $K_n$ be the N\'eron-Severi lattice of the Kummer surface associated to a K3 surface with Shioda-Inose structure and N\'eron-Severi lattice $M_n$. 

Suppose now that we can undo the Kummer construction for $\calY_n$, by pulling back to a cover $C_{M_n} \to \calM_{M_n}$. Then if we know that an $M_n$-polarised family of K3 surfaces $\mathcal{X} \to U$ is the pull-back of a family $\calX_n \to U_{M_n}$ by the period map $U \to \calM_{M_n}$ (which, in the $M_n$-polarized case, is more commonly known as the \emph{generalized functional invariant}, see \cite{pfummmmm}), then we can undo the Kummer construction for the associated family of Kummer surfaces $\calY \to U$ by pulling back to the fibre product $U \times_{\calM_{M_n}} C_{M_n}$.

Thus the aim of this section is to find covers $C_{M_n} \to \calM_{M_n}$ such that the pull-backs of $\calY_n$ to $C_{M_n}$ are $K_n$-polarized (and so, by Proposition \ref{prop:untwistingkummer}, the Kummer construction can be undone on these pull-backs).

\begin{lemma} \label{lemma:MnG} The families $\calY_n$ are minimally $(K_n,G)$-polarized, where $G$ is the group $G = \mathrm{O}(M_n^{\perp})^*/ \mathrm{O}(M_n^{\perp}(2))^*$
\end{lemma}
\begin{proof} This follows from Proposition \ref{prop:generalcaseG} and the fact that the families $\calX_n$ have transcendental monodromy groups $\mathrm{O}(M_n^\perp)^*$.
\end{proof}

As $\calM_{M_n} = \mathrm{O}(M_n^\perp)^* \setminus \mathcal{P}_{M_n}$, this lemma suggests that, in order to undo the action of $G$, we should define $C_{M_n}$ to be the curve $C_{M_n} := \mathrm{O}(M_n^{\perp}(2))^* \setminus \mathcal{P}_{M_n}$. This curve may be constructed as a modular curve in the following way.

Recall that 
$$
\Gamma_0(n) := \left\{ \gamma \in \text{SL}_2(\mathbb{Z}) : \gamma \equiv \left( \begin{matrix} * & * \\ 0 & * \end{matrix} \right) \bmod n \right\}
$$
and 
$$
\Gamma(n) := \left\{ \gamma \in \text{SL}_2(\mathbb{Z}) : \gamma \equiv \left( \begin{matrix} 1 & 0 \\ 0 & 1 \end{matrix} \right) \bmod n \right\}.
$$
By convention, $\Gamma_0(1)$ and $\Gamma(1)$ are just the full modular group $\Gamma = \text{SL}_2(\mathbb{Z})$. We also have  
$$
\Gamma_0(n)^+:=  \Gamma_0(n) \cup \tau_n \Gamma_0(n) \subseteq \text{SL}_2(\mathbb{R})
$$ 
where 
$$
\tau_n = \left( \begin{matrix} 0 & -1/\sqrt{n} \\ \sqrt{n} & 0 \end{matrix} \right)
$$
is the \emph{Fricke involution}. With this notation, we have $\calM_{M_n} \cong \Gamma_0(n)^+ \setminus \mathbb{H}$ \cite[Theorem 7.1]{mslpk3s}.

For any lattice $N$, let $\mathbb{P}\text{O}(N)$ be defined as the cokernel of the obvious injection $\pm \Id \hookrightarrow \text{O}(N)$. Then we have the exact sequence
$$
1 \rightarrow \{ \pm \Id \} \rightarrow \mathrm{O}(N) \rightarrow \mathbb{P} \text{O}(N) \rightarrow 1.
$$
If $N$ is a lattice of signature $(1,n-1)$ with a fixed primitive embedding into $\Lambda_{\text{K3}}$ and $\Gamma$ and $\Gamma'$ are two subgroups of $\text{O}(N^\perp)$, the quotients $\Gamma \setminus \mathcal{P}_N$ and $\Gamma' \setminus \mathcal{P}_N$ are the same if and only if $\Gamma$ and $\Gamma'$ have the same images in $\mathbb{P}\text{O}(N^\perp)$, in which case $\Gamma$ and $\Gamma'$ are said to be \emph{projectively equivalent}.

By \cite[Theorem 7.1]{mslpk3s}, there is a map $R_n$, defined in the following proposition, under which $\Gamma_0(n)^+$ is mapped to a subgroup of $\SO(M_n^\perp)$ that is projectively equivalent to $\text{O}(M_n^\perp)^*$.

\begin{lemma} \label{lemma:modular}
The group $\mathrm{O}(M_n^{\perp}(2))^*$ is projectively equivalent to the image of $\Gamma(2) \cap \Gamma_0(2 n)$ under the map 
$$
R_n: \mathrm{SL}_2(\mathbb{R}) \rightarrow \mathrm{SO}_\mathbb{R}(2,1)
$$
which is defined as 
$$
\left( \begin{matrix} a & b \\ cn & d \end{matrix} \right) \mapsto \left(\begin{array}{rrr}
a^{2} & c^{2}n & 2 a c n \\
b^{2} n & d^{2} & 2 b d n \\
a b & c d & b cn + a d
\end{array}\right)
$$
\textup{(}see the related map in \cite[Equation 5.6]{adck3smt}\textup{)}.
\end{lemma}
\begin{proof}
We know that the pre-image of $\mathrm{O}(M_n^\perp)^*$ under $R_n$ is the subgroup $\Gamma_0(n)^+$ and that $\text{O}(M_n^\perp(2))^* \subseteq \text{O}(M_n^\perp)^*$ is the subgroup which fixes the group $A_{M_n^\perp(2)}$. Since $R_n$ maps the Fricke involution to the automorphism
$$
\left( \begin{matrix} 0 & 1 & 0 \\ 1 & 0 & 0 \\ 0 & 0 & -1 \end{matrix} \right),
$$
which is never trivial or $-\Id$ on $A_{M_n^\perp(2)}$, we may automatically restrict to the image of $\Gamma_0(n)$. Automorphisms which fix $A_{M_n^\perp(2)}$ are matrices of the form
$$
\left( \begin{matrix} a_{11} & a_{12} & a_{13} \\a_{21} & a_{22} & a_{23} \\a_{31} & a_{32} & a_{33}  \end{matrix} \right)
$$
with $a_{12}, a_{21},a_{31}, a_{32} \equiv 0 \bmod 2$, $a_{13}, a_{23}, \equiv 0 \bmod 2n$, $a_{11}, a_{22} \equiv 1 \bmod 2$ and $a_{33} \equiv 1 \bmod 2n$. Thus $a^2 \equiv d^2 \equiv 1 \bmod 2$ and hence $a,d \equiv 1 \bmod 2$. Using this and the fact that $ab \equiv cd \equiv 0 \bmod 2$, we find that $b \equiv c \equiv 0 \bmod 2$. Therefore the matrices which map to $\text{O}(M_n^\perp(2))^*$ are precisely those which satisfy
$$
\left( \begin{matrix} a & b \\ cn & d \end{matrix} \right)  \equiv \left( \begin{matrix} * & * \\ 0 & * \end{matrix} \right)  \bmod 2n
$$
and
$$
\left( \begin{matrix} a & b \\ cn & d \end{matrix} \right)  \equiv \left( \begin{matrix} 1 & 0 \\ 0 &  1 \end{matrix} \right)  \bmod 2.
$$
In other words elements of the group $\Gamma_0(2n) \cap \Gamma(2)$.
\end{proof}

We therefore have
\[C_{M_n} \cong (\Gamma_0(2n) \cap \Gamma(2)) \setminus \HH.\]
Let $f\colon C_{M_n}\rightarrow \mathcal{M}_{M_n}$ be the natural map coming from the modular description of each curve. 
\begin{proposition}
If $n \neq 1$, the pullback $f^*\mathcal{Y}_n$ of $\calY_n$ to $C_{M_n}$ is $K_n$-polarized.
\end{proposition}
\begin{proof}
The transcendental monodromy of the pullback $f^* \mathcal{X}_n$ is a group $\Gamma$ contained in $\text{O}(M_n^\perp)^*$ with quotient space $\Gamma \setminus \mathcal{P}_{M_n} \cong (\Gamma_0(2n) \cap \Gamma(2))\setminus \mathbb{H}$. By Lemma \ref{lemma:modular}, the group $\text{O}((M_n^\perp)(2))^*$ has this property.

Suppose that there is another subgroup $\Gamma'$ of $\text{O}(M_n^\perp)^*$ with this property. Let $\gamma \in \Gamma$ be any element and let $g \in \mathbb{P}\text{O}(M_n^\perp)$ be its image. Since $\Gamma$ and $\Gamma'$ are projectively equivalent, there is some $\gamma' \in \Gamma'$ which maps to $g$.

If $\Gamma$ and $\Gamma'$ are not the same group, we can find some $g \in \mathbb{P}\text{O}(M_n^\perp)$ such that there are $\gamma \in \Gamma$ and $\gamma' \in \Gamma$ which map to $g$ yet have $\gamma \neq \gamma'$. Thus $\gamma^{-1} \gamma' \neq \Id$ but $\gamma^{-1} \gamma'$ maps to the identity in $\mathbb{P}\text{O}(M_n^\perp)$. However, for $n \neq 1$, \cite[Lemma 1.15]{adck3smt} shows that the kernel of $\mathrm{O}(M_n^\perp)^* \rightarrow \mathbb{P}\text{O}(M_n)$ is trivial. This is a contradiction, hence $\Gamma = \Gamma'$. 

Therefore, the monodromy group of the family $f^*\mathcal{X}_n$ is $\text{O}(M_n^\perp(2))^* \subseteq \text{O}(M_n^\perp)^*$. By Corollary \ref{cor:Gfind}, the associated family of Kummer surfaces then has transcendental monodromy $\text{O}(M_n^\perp(2))^*$ as well. Since this group is contained in the kernel of $\alpha_{\T(\calY_n)}$, we conclude that $\mathcal{Y}_n$ is $K_n$-polarized.
\end{proof}

\begin{remark} This discussion may be rephrased in the following way. The quotient-resolution procedure taking $\calX_n$ to $\calY_n$ defines an isomorphism $\calM_{M_n} \stackrel{\sim}{\longrightarrow} \calM_{(K_n,G)}$, where $G$ is the group from Lemma \ref{lemma:MnG}. The cover $C_{M_n} \to \calM_{M_n}$ is then precisely the cover $\calM_{K_n} \to \calM_{(K_n,G)}$.
\end{remark}

In the case where $n=1$ this proof fails, as the kernel of the map $\mathrm{O}(M_n^\perp)^* \rightarrow \mathbb{P}\text{O}(M_n)$ is nontrivial. It will therefore be necessary for us to do a little more work in order to find a cover of $\mathcal{M}_{M_1}$ on which the pullback of $\mathcal{Y}_1$ is lattice polarized. 

The family $\mathcal{X}_1$ is a family of smooth K3 surfaces over $\mathbb{P}^1 \setminus \{0,1, \infty\}$. Let $g_1$ and $g_2$ in $\mathrm{O}(H\oplus\langle 2\rangle)^*$ be as in the $n=1$ case of the proof of Proposition \ref{MnG}: then $g_1$ describes monodromy around $1$ and $g_2$ describes monodromy around $\infty$, and monodromy around $0$ is, as usual, given by $g_1 g_2^{-1}$. Around the point $1$, the order of monodromy is $2$, around $0$, the order of monodromy is $6$, and around $\infty$, the order of monodromy is infinite.

The group $\Gamma_0(2) \cap \Gamma(2)$ is just $\Gamma(2)$, since $\Gamma(2) \subseteq \Gamma_0(2)$, and the map from $C_{M_1} = \Gamma(2)\setminus \HH$ to $\calM_{M_1} = \Gamma_0(1)^+\setminus \HH \cong \Gamma_0(1) \setminus \HH$ is just the $j$-function of the Legendre family of elliptic curves. This map may be written as a rational function, 
$$
j(t) = \dfrac{(t^{2} - t + 1)^{3}}{27 t^{2} (t - 1)^{2} }.
$$
The function $j(t)$ has three ramification points of order $2$ over $1$, three ramification points of order $2$ over $\infty$ and two ramification points points of order $3$ over $0$. Looking back at the proof of Proposition \ref{MnG}, we see that the monodromy around the preimages of $1$ and $\infty$ must act as $h_1^2 = \Id$ and $h_2^2 = \Id$ on $A_{H(2)\oplus \langle 4\rangle}$. However, monodromy around the preimages of $0$ acts on $A_{H(2)\oplus \langle 4\rangle}$ as $(h_1h_2)^2 = -\Id$. Therefore, in order for monodromy to act trivially on $A_{H(2)\oplus \langle 4\rangle}$, we must take a further double cover of $C_{M_1} = \Gamma(2) \setminus \mathbb{H} = \mathbb{P}^1_t$ ramified along the roots of $t^{2} - t + 1 = 0$. We thus have:

\begin{proposition} \label{prop:M1}
If $n=1$, there is a double cover $C_{M_1}'$ of $C_{M_1}$ on which the pull-back of the family $\mathcal{Y}_n$ is $K_1$-polarized.
\end{proposition}

The maps $f\colon C_{M_n} \to \calM_{M_n}$ will be calculated in the next section.

\subsection{Covers for small \texorpdfstring{$n$}{n}} \label{section:covers}

In this section, we will explicitly compute the maps $f \colon C_{M_n} \to \calM_{M_n}$ for $n \leq 4$. To do this, we decompose the map $f = f_1 \circ f_2 \circ f_3$, where
\begin{align*}
f_1 \colon &  \Gamma_0(n) \setminus \HH \longrightarrow \Gamma_0(n)^+ \setminus \HH, \\
f_2 \colon & \Gamma_0(2n) \setminus \HH \longrightarrow \Gamma_0(n) \setminus \HH,  \\
f_3 \colon &  C_{M_n} \cong (\Gamma_0(2n) \cap \Gamma(2)) \setminus \HH \longrightarrow \Gamma_0(2n) \setminus \HH.
\end{align*}

\subsubsection{The case \texorpdfstring{$n=1$}{n=1}} The rational modular curves $\Gamma_0(1)^+ \setminus \HH$ and $\Gamma_0(1)\setminus \HH$ are isomorphic and have two elliptic points of orders $2$ and $3$ along with a single cusp. The map $f_2$ is a triple cover ramified with index $3$ over the elliptic point of order $2$ and indices $(2,1)$ over the elliptic point of order $2$ and the cusp. $\Gamma_0(2) \setminus \HH$ is a rational modular curve with an elliptic point of order $2$ and two cusps. Finally, $f_3$ is a double cover ramified over the elliptic point and the cusp that is not a ramification point of $f_2$ and $C_{M_1}$ is a rational modular curve with three cusps.

We thus see that $f\colon C_{M_1} \to \Gamma_0(1)^+ \setminus \HH$ is a $6$-fold cover ramified with indices $2$ and $3$ at all points over the elliptic points of order $2$ and $3$ respectively and index $2$ at all points over the cusp. It is easy to see that the deck transformation group of $f$ is $S_3$.

However, from Proposition \ref{prop:M1}, we need to take a further double cover of $C_{M_1}$ before  we can undo the Kummer construction. This double cover is ramified over the two preimages under $f$ of the elliptic point of order $3$. The composition $C_{M_1}' \to \Gamma_0(1)^+ \setminus \HH$  is a $12$-fold cover ramified with indices $2$ and $6$ at all points over the elliptic points of order $2$ and $3$ respectively and index $2$ at all points over the cusp. It is easy to see that the deck transformation group of this composition is $S_3 \times C_2$, as expected from Proposition \ref{MnG}.

\subsubsection{The case \texorpdfstring{$n=2$}{n=2}} The rational modular curve $\Gamma_0(2)^+ \setminus \HH$ has two elliptic points of orders $2$ and $4$ and a single cusp. The map $f_1$ is a double cover ramified over the two elliptic points and $\Gamma_0(2) \setminus \HH$  is a rational modular curve with a single elliptic point of order $2$ and two cusps. The map $f_2$ is then a double cover ramified over the elliptic point and one of the cusps, and $\Gamma_0(4) \setminus \HH$ is a rational modular curve with three cusps. Finally, $f_3$ is a double cover ramified over the two cusps that are not ramification points of $f_2$ and $C_{M_2}$ is a rational modular curve with four cusps.

We thus see that $f\colon C_{M_2} \to \Gamma_0(2)^+ \setminus \HH$ is an $8$-fold cover ramified with indices $2$ and $4$ at all points over the elliptic points of order $2$ and $4$ respectively and index $2$ at all points over the cusp. It is easy to see that the deck transformation group of $f$ is $D_8$, as expected from Proposition \ref{MnG}.

\subsubsection{The case \texorpdfstring{$n=3$}{n=3}} The rational modular curve $\Gamma_0(3)^+ \setminus \HH$ has two elliptic points of orders $2$ and $6$ and a single cusp. The map $f_1$ is a double cover ramified over the two elliptic points and $\Gamma_0(3) \setminus \HH$  is a rational modular curve with one elliptic point of order $3$ and two cusps. The map $f_2$ is then a triple cover ramified with index $3$ over the elliptic point and indices $(2,1)$ over each of the cusps, and $\Gamma_0(6) \setminus \HH$ is a rational modular curve with four cusps. Finally, $f_3$ is a double cover ramified over the two cusps that are not ramification points of $f_2$ and $C_{M_3}$ is a rational modular curve with six cusps.

We thus see that the map $f\colon C_{M_3} \to \Gamma_0(3)^+ \setminus \HH$ is an $12$-fold cover ramified with indices $2$ and $6$ at all points lying over the elliptic points of orders $2$ and $6$ respectively and index $2$ at all points over the cusp. It is easy to see that the deck transformation group of $f$ is $D_{12}$, as expected from Proposition \ref{MnG}.

\subsubsection{The case \texorpdfstring{$n=4$}{n=4}} The rational modular curve $\Gamma_0(4)^+ \setminus \HH$ has an elliptic point of order $2$ and two cusps. The two cusps are distinguished by their widths, which are $1$ and $2$. The map $f_1$ is a double cover ramified over the elliptic point and the cusp of width $2$. The rational modular curve $\Gamma_0(4) \setminus \HH$  has three cusps of widths $(4,1,1)$. The map $f_2$ is then a double cover ramified with index $2$ over the cusp of width $4$ and one of the cusps of width $1$. The rational modular curve $\Gamma_0(8) \setminus \HH$ has four cusps of widths $(8,2,1,1)$.  Finally, $f_3$ is a double cover ramified over the two cusps of width $1$. The curve $C_{M_4}$ is a rational modular curve with six cusps of widths $(8,8,2,2,2,2)$. 

We thus see that $f\colon C_{M_4} \to \Gamma_0(4)^+ \setminus \HH$ is an $8$-fold cover ramified with index $2$ at all points lying over the elliptic point and indices $2$ and $4$ at all points over the cusps of widths $1$ and $2$ respectively. It is easy to see that the deck transformation group of $f$ is $D_{8}$, as expected from Proposition \ref{MnG}.

\begin{remark} Note that if $n \neq 1$ we may also find a cover of $\calY_n \to U_{M_n}$ that is $K_n$-polarized using the method of Section \ref{undokummer} (if $n = 1$ then this method cannot be used, as Assumption \ref{I2ass} fails; see Section \ref{sect:M1}). In the three cases with $n \geq 2$ above it may be seen that this cover agrees with $C_{M_n}$. \end{remark}

\subsection{Application to the 14 cases.}\label{sect:14caseapp}

We now apply this theory to undo the Kummer construction for families of Kummer surfaces arising from $M$-polarized fibrations on the fourteen cases in \cite[Table 1]{doranmorgan}. 

Examining these cases, we find $M_n$-polarized K3 fibrations with $2 \leq n \leq 4$ on nine of them, listed in the appropriate sections of Table \ref{table:12cases}. In this table, the first column gives the polarization lattice $M$ or $M_n$, the second gives the mirrors of the threefolds that have $M$- or $M_n$-polarized K3 fibrations, and the third states whether or not these fibrations are torically induced (the meanings of the fourth and fifth columns will be discussed later). More precisely, we have:

\begin{theorem} \label{thm:Mnfamilies} There exist K3 fibrations with $M_n$-polarized generic fibre, for $2 \leq n \leq 4$, on nine of the threefolds in \cite[Table 1]{doranmorgan}, given by the mirrors of those listed in the appropriate sections of Table \ref{table:12cases}. Furthermore, if $\calX \to \Proj^1$ denotes one of these fibrations and $U \subset \Proj^1$ is the open set over which the fibres of $\calX$ are nonsingular, then the restriction $\calX|_U\to U$ agrees with the pull-back of a family $\calX_n$ \textup{(}see Definition \ref{Xndefn}\textup{)} by the generalized functional invariant map $U \to \calM_{M_n}$. The family $\calX|_U\to U$ is thus an $M_n$-polarized family of K3 surfaces.  \end{theorem}

\begin{table}
\begin{tabular}{|c|c|c|c|c|}
\hline 
Lattice & Mirror threefold & Toric? & $(i,j)$ & Arithmetic/thin \\
\hline
 & $\WP(1,1,1,1,2)[6]$ & Yes & $(1,2)$ & Arithmetic \\
& $\WP(1,1,1,1,4)[8]$ & Yes & $(1,3)$ &  Thin \\ 
$M_1$ &$\WP(1,1,1,2,5)[10]$ &Yes & $(2,3)$ &  Arithmetic \\
& $\WP(1,1,1,1,1,3)[2,6]^*$ & Yes & $(1,1)$ &  Thin \\
& $\WP(1,1,1,2,2,3)[4,6]^*$ & Yes & $(2,2)$ &  Arithmetic \\
\hline
& $\Proj^4[5]$ & Yes & $(1,4)$ & Thin \\
& $\WP(1,1,1,1,2)[6]$ & Yes & $(2,4)$ &  Arithmetic \\
$M_2$ & $\WP(1,1,1,1,4)[8]$ & Yes & $(4,4)$ &  Thin \\ 
&$\Proj^5[2,4]$ & Yes & $(1,1)$ & Thin \\
& $\WP(1,1,1,1,2,2)[4,4]$ & Yes & $(2,2)$ &  Arithmetic \\
\hline
& $\Proj^4[5]$ & No & $(2,3)$ &  Thin \\
& $\Proj^5[2,4]$ & No & $(1,3)$ & Thin \\
$M_3$ & $\Proj^5[3,3]$ & Yes & $(1,2)$ &  Arithmetic \\
 & $\WP(1,1,1,1,1,2)[3,4]^*$ & Yes & $(2,2)$ &  Arithmetic \\ 
& $\Proj^6[2,2,3]$ & Yes & $(1,1)$ &  Thin \\
\hline
& $\Proj^5[2,4]$ & No & $(2,2)$ & Thin \\
$M_4$ & $\Proj^6[2,2,3]$ & No & $(1,2)$ & Thin \\
& $\Proj^7[2,2,2,2]$ & Yes & $(1,1)$ &  Thin \\
\hline
$M$ & $\WP(1,1,1,1,4,6)[2,12]$ & Yes & $(1,1)$ & Thin \\
\hline
\end{tabular}
\caption{Lattice polarized K3 fibrations on the threefolds from \cite[Table 1]{doranmorgan}.}
\label{table:12cases}
\end{table}

\begin{remark} The $M_1$-polarized cases in the first section of Table \ref{table:12cases} will require some extra work, so they will be discussed separately in Section \ref{sect:M1}. The 14th case of \cite{14thcase} has already been discussed in Example \ref{ex:14thcase}, where we recalled that the family of threefolds $Y_1$ realising the 14th case variation of Hodge structure admit torically induced $M$-polarized K3 fibrations. By \cite[Section 8.2]{14thcase}, these threefolds $Y_1$ can be thought of as mirror to complete intersections $\WP(1,1,1,1,4,6)[2,12]$. This case is included in the final row of Table \ref{table:12cases}. 
\end{remark}

\begin{remark} To check which of the fibrations listed in Table \ref{table:12cases} are torically induced, one may use the computer software \emph{Sage} to find all fibrations of the toric ambient spaces by toric subvarieties that induce fibrations of the Calabi-Yau threefold by $M$-polarized K3 surfaces. The resulting list may be compared to the list of fibrations in Table \ref{table:12cases}, giving the third column of this table. This also proves that Table \ref{table:12cases} contains \emph{all} torically induced fibrations of the Calabi-Yau threefolds from \cite[Table 1]{doranmorgan} by $M$-polarized K3 surfaces.
\end{remark}

We will prove Theorem \ref{thm:Mnfamilies} by explicit calculation: we find families $\calX_n$ satisfying Definition  \ref{Xndefn} and show that they pull back to give the families $\calX|_U$ under the generalized functional invariant maps. 

In each case, we will see that the generalized functional invariant map is completely determined by the pair of integers $(i,j)$ from the fourth column of Table \ref{table:12cases}. In fact, we find that it is an $(i+j)$-fold cover of $\calM_{M_n} \cong \Gamma_0(n)^+ \setminus \HH$ having exactly four ramification points: one of order $(i+j)$ over the cusp (or, in the $M_4$-polarized case, the cusp of width $1$), two of orders $i$ and $j$ over the elliptic point of order $\neq 2$ (or, in the $M_4$-polarized case, the cusp of width $2$), and one of order $2$ which varies with the value of the Calabi-Yau deformation parameter.

We thus have everything we need to undo the Kummer construction in the families arising as the resolved quotients of the families $\calX|_U$ from Theorem \ref{thm:Mnfamilies}. By the discussion in Section \ref{section:specialfamilies}, in order to undo the Kummer construction we just need to pull
 back to the cover $C_{M_n} \times_{\calM_{M_n}} U$, where the map $C_{M_n} \to \calM_{M_n}$ is as calculated in Section \ref{section:covers} and $U \to \calM_{M_n}$ is the generalized functional invariant map, described above.

\subsubsection{\texorpdfstring{$M_2$}{M2}-polarized families} We begin the proof of Theorem \ref{thm:Mnfamilies} with the $M_2$ case. Note first that an $M_2$-polarized K3 surface is mirror (in the sense of \cite{mslpk3s}) to a $\langle 4 \rangle$-polarized K3 surface, which is generically a hypersurface of degree $4$ in $\Proj^3$.

By the Batyrev-Borisov mirror construction \cite{cycitv}, the mirror of a degree $4$ hypersurface in $\mathbb{P}^3$ is a hypersurface in the toric variety polar dual to $\Proj^3$. The intersection of this hypersurface with the maximal torus is isomorphic to the locus in $(\mathbb{C}^\times)^3$ defined by the rational polynomial
\begin{equation} \label{eq:M2} x_1 + x_2 + x_3 + \frac{\lambda}{x_1x_2x_3} = 1,\end{equation}
where $\lambda \in \C$ is a constant. This is easily compactified to a singular hypersurface of degree $4$ in $\mathbb{P}^3$, given by the equation
\[\lambda w^4 +  xyz(x+y+z-w) = 0,\]
where $(w,x,y,z)$ are coordinates on $\Proj^3$. 

Consider the family of surfaces over $\C$ obtained by varying $\lambda$. By resolving the singularities of the generic fibre and removing any singular fibres that remain, we obtain a family of K3 surfaces $\calX_2 \to U_{2} \subset \C$. Dolgachev \cite[Example (8.2)]{mslpk3s} exhibited elliptic fibrations on the K3 fibres of $\calX_2$ and used them to give a set of divisors generating the lattice $M_2$. It can be seen from the structure of these elliptic fibrations that these divisors are invariant under monodromy, so there can be no action of monodromy on $M_2$. We thus see that $\calX_2$ is an $M_2$-polarized family of K3 surfaces.

The action of transcendental monodromy on $\calX_2$ was calculated by Narumiya and Shiga \cite{mmfk3sis3drp} (note that our parameter $\lambda$ is different from theirs: our $\lambda$ is equal to $\mu^4$ or $\frac{u}{256}$ from their paper). In \cite[Section 4]{mmfk3sis3drp} they find that the fibre $X_{\lambda}$ of $\calX_2$ is smooth away from $\lambda \in \{0,\frac{1}{256}\}$ and the monodromy action has order $2$ around $\lambda = \frac{1}{256}$, order $4$ around $\lambda = \infty$, and infinite order around $\lambda = 0$. Furthermore, they show \cite[Remark 6.1]{mmfk3sis3drp} that the monodromy of $\calX_2$ generates the $(2,4,\infty)$ triangle group (which is isomorphic to $\Gamma^0(2)^+ \cong \mathrm{O}(M_2^\perp)^*$), so the period map of $\mathcal{X}_2 \to U_{2}$ must be injective. Thus the family $\calX_2 \to U_{2}$ satisfies Definition \ref{Xndefn}.

We can use the local form \eqref{eq:M2} of the family $\calX_2$ to find $M_2$-polarized families of K3 surfaces on the threefolds from \cite[Table 1]{doranmorgan}. For example:

\begin{ex} The first $M_2$-polarized case from Table \ref{table:12cases} is the mirror to the quintic threefold. By the Batyrev-Borisov construction, on the maximal torus we may write this mirror as the  locus in $(\mathbb{C}^\times)^4$ defined by the rational polynomial
\[ x_1 + x_2 + x_3 + x_4 + \frac{A}{x_1x_2x_3x_4} = 1,\]
where $A \in \C$ is the Calabi-Yau deformation parameter. Consider the fibration induced by projection onto the $x_4$ coordinate; for clarity, we make the substitution $x_4 = t$. If we further substitute $x_i \mapsto x_i(1-t)$ for $1 \leq i \leq 3$ and rearrange, we obtain
\[ x_1 + x_2 + x_3 + \frac{A}{x_1x_2x_3t(1-t)^4} = 1.\]
But, from the local form \eqref{eq:M2}, it is clear that this describes an $M_2$-polarized family of K3 surfaces with 
\[ \lambda = \frac{A}{t(1-t)^4}.\]
This is the generalized functional invariant map of the fibration. Note that it is ramified to orders $1$ and $4$ over the order $4$ elliptic point $\lambda = \infty$, order $5$ over the cusp $\lambda = 0$, and order $2$ over the variable point $\lambda = \frac{5^5A}{2^8}$, giving $(i,j) = (1,4)$.
\end{ex}

Similar calculations may be performed in the other $M_2$-polarized cases from Table \ref{table:12cases}. We find that the generalized functional invariants are given by
\[ \lambda = \frac{Au^{i+j}}{t^i(u-t)^j},\]
where $(t,u)$ are homogeneous coordinates on the base $U \subset \Proj^1$ of the K3 fibration, $(i,j)$ are as in Table \ref{table:12cases}, and $A$ is the Calabi-Yau deformation parameter.

\subsubsection{\texorpdfstring{$M_3$}{M3}-polarized families} Here we follow a similar method to the $M_2$-polarized case. An $M_3$-polarized K3 surface is mirror to a $\langle 6 \rangle$-polarized K3 surface, which may be realised as a complete intersection of type $(2,3)$ in $\Proj^4$.

By the Batyrev-Borisov construction, on the maximal torus we may express the mirror of a $(2,3)$ complete intersection in $\mathbb{P}^4$ as the locus in $(\mathbb{C}^\times)^3$ defined by the rational polynomial
\begin{equation} \label{eq:M3} x_1 + \frac{\lambda}{x_1x_2x_3(1-x_2-x_3)} = 1,\end{equation}
where $\lambda \in \C$ is a constant. This is easily compactified to a singular hypersurface of bidegree $(2,3)$ in $\mathbb{P}^1 \times \mathbb{P}^2$, given by the equation
\[\lambda s^2 z^3 +  r(r-s)xy(z-x-y) = 0,\]
where $(r,s)$ are coordinates on $\Proj^1$ and $(x,y,z)$ are coordinates on $\Proj^2$. 

Consider the family of surfaces over $\C$ obtained by varying $\lambda$. By resolving the singularities of the generic fibre and removing any singular fibres that remain, we obtain a family of K3 surfaces $\calX_3 \to U_{3} \subset \C$. We now show that $\calX_3$ is an $M_3$-polarized family that satisfies Definition \ref{Xndefn}.

There is a natural elliptic fibration on the fibres of $\calX_3$, obtained by projecting onto the $\mathbb{P}^1$ factor. This elliptic fibration has two singular fibers of Kodaira type $IV^*$ at $r=0$ and $r=s$, a fibre of type $I_6$ at $s=0$, two fibres of type $I_1$ and a section. In fact, one sees easily that the hypersurface obtained by intersecting with $z=0$ splits into three lines, which project with degree $1$ onto $\mathbb{P}^1$ and hence are all sections. If we choose one of these sections as a zero section, the other two are $3$-torsion sections and generate a subgroup of the Mordell-Weil group of order $3$.

One can check that the lattice spanned by components of reducible fibers and these torsion sections is a copy of the lattice $M_3$ inside of $\NS(X_{\lambda})$, for each fiber $X_{\lambda}$ of $\mathcal{X}_3 \to U_{3}$. Since the $3$-torsion sections are individually fixed under monodromy, there can be no monodromy action on this copy of $M_3$ in $\NS(X_{\lambda})$. We thus see that $\calX_3$ is an $M_3$-polarized family of K3 surfaces.

Next we calculate the transcendental monodromy of this family to show that it satisfies Definition \ref{Xndefn}.

\begin{lemma}
\label{lemma:M3monodromy}
$U_{3}$ is the open subset given by removing the points $\lambda \in \{0,\frac{1}{108}\}$ from $\C$. Transcendental monodromy of the family $\mathcal{X}_3 \to U_{3}$ has order $2$ around $\lambda = \frac{1}{108}$, order dividing $6$ around $\lambda = \infty$ and infinite order around $\lambda = 0$.
\end{lemma}
\begin{proof} The discriminant of the elliptic fibration on a fibre $X_{\lambda}$ of $\calX_3$ vanishes for $\lambda \in \{0, \frac{1}{108},\infty\}$, giving the locations of the singular K3 surfaces that are removed from the family $\calX_3$. At $\lambda = \frac{1}{108}$ the two singular fibres of type $I_1$ collide so that the K3 surface $X_{\lambda = \frac{1}{108}}$ has a single node. Thus there is a vanishing class of square $(-2)$ associated to the fibre $X_{\lambda = \frac{1}{108}}$ and monodromy around this fibre is a reflection across this class. Therefore monodromy around $\lambda = \frac{1}{108}$ has order $2$.

We will use this to indirectly calculate the monodromies around other points. After base change $\lambda = \mu^3$ and a change in variables, one finds that the $\lambda = \infty$ fiber can be replaced with an elliptically fibered K3 surface with three singular fibers of type $IV^*$. Since a generic member of the family $\mathcal{X}_3$ has N\'eron-Severi rank $19$, this fiber can only have a single node, so again the monodromy transformation around it must be of order at most $2$. Hence monodromy around $\lambda =\infty$ has order dividing $6$.

To determine monodromy around the final point, it is enough to note that the moduli space of $M_n$-polarized K3 surfaces has a cusp, and the preimage of this cusp under the period map must also have monodromy of infinite order. Since the points $\lambda \in \{\frac{1}{108},\infty\}$ are of finite order and every other fiber is smooth, $\lambda = 0$ must map to the cusp under the period map and therefore has infinite order monodromy.
\end{proof}

As a result we find:

\begin{proposition}
The period map of $\mathcal{X}_3 \to U_{3}$ is injective and the subgroup of $\mathrm{O}(M_3^\perp)^*$ generated by monodromy transformations is $\mathrm{O}(M_3^\perp)^*$ itself. The family $\calX_3$ thus satisfies Definition \ref{Xndefn}.
\end{proposition}
\begin{proof}
Notice first that, by Lemma \ref{lemma:M3monodromy}, the monodromy group of $\mathcal{X}_3$ is isomorphic to a triangle group of type $(2,d,\infty)$ for $d = 2,3$ or $6$ and contained in $\mathrm{O}(M_3^\perp)^*$. It is well known that $\mathrm{O}(M_3^\perp)^* \cong \Gamma_0(3)^+$ is a $(2,6,\infty)$ triangle group, and since the period map is of finite degree, the monodromy group of $\mathcal{X}_3$ is of finite index in $\Gamma_0(3)^+$. Thus we need to show that the only finite index embedding of a $(2,d,\infty)$ triangle group into the $(2,6,\infty)$ triangle group is the identity map from the $(2,6,\infty)$ triangle group to itself. But this is calculated in \cite{ccatg}.
\end{proof}

As before, we can use the local form \eqref{eq:M3} of the family $\calX_3$ to find $M_3$-polarized families of K3 surfaces on the threefolds from \cite[Table 1]{doranmorgan}. We find that the generalized functional invariants are given by
\[ \lambda = \frac{Au^{i+j}}{t^i(u-t)^j},\]
where $(t,u)$ are homogeneous coordinates on the base $U \subset \Proj^1$ of the K3 fibration, $(i,j)$ are as in Table \ref{table:12cases}, and $A$ is the Calabi-Yau deformation parameter.

\subsubsection{\texorpdfstring{$M_4$}{M4}-polarized families} We conclude the proof of Theorem \ref{thm:Mnfamilies} with the $M_4$-polarized case. An $M_4$-polarized K3 surface is mirror to an $\left<8 \right>$-polarized K3 surface, given generically as a complete intersection of type $(2,2,2)$ in $\mathbb{P}^5$. 

By the Batyrev-Borisov construction, on the maximal torus we may express the mirror of a complete intersection of type $(2,2,2)$ in $\Proj^5$ as the locus in $(\C^{\times})^3$ defined by the rational polynomial
\begin{equation} \label{eq:M4} x_1 + \frac{\lambda}{x_2(1-x_2)x_3(1-x_3)x_1} = 1.\end{equation}
This may be easily compactified to a singular hypersurface of multidegree $(2,2,2)$ in $(\mathbb{P}^1)^3$ given by
\[ \lambda s_1^2s_2^2s_3^2 - r_1(s_1-r_1)r_2(s_2-r_2)r_3(s_3-r_3) = 0,\]
where $(r_i,s_i)$ are coordinates on the $i$th copy of $\Proj^1$.

As above, we consider the family of surfaces over $\C$ obtained by varying $\lambda$. By resolving the singularities of the generic fibre and removing any singular fibres that remain, we obtain a family of K3 surfaces $\calX_4 \to U_{4} \subset \C$. We now show that $\calX_4$ is an $M_4$-polarized family that satisfies Definition \ref{Xndefn}.

Begin by noting that there is an $S_3$ symmetry on $\calX_4$ obtained by permuting copies of $\mathbb{P}^1$. Furthermore, projection of $(\mathbb{P}^1)^3$ onto any one of the three copies of $\mathbb{P}^1$ produces an elliptic fibration on the K3 hypersurfaces. This elliptic fibration has a description very similar to that of the elliptic fibration on $\mathcal{X}_3$. Generically it has two fibres of type $I_1^*$ at $r_i = 0$ and $r_i = s_i$, a fibre of type $I_8$ at $s_i = 0$, and two fibres of type $I_1$.

This elliptic fibration has a $4$-torsion section. Using standard facts relating the N\'eron-Severi group of an elliptic fibration to its singular fiber types and Mordell-Weil group (see \cite[Lecture VII]{btes}), we see that each fiber of $\mathcal{X}_4$ is polarized by a rank $19$ lattice with discriminant $8$. A little lattice theory shows that this must be the lattice $M_4$. The embedding of $M_4$ into the N\'{e}ron-Severi group must be primitive, otherwise we would find full $2$-torsion structure, which is not the case. As in the case of $\calX_3$, this embedding of $M_4$ is monodromy invariant, so $\mathcal{X}_4$ is an $M_4$-polarized family of K3 surfaces.

\begin{proposition}
\label{prop:M4monodromy}
$U_{4}$ is the open subset given by removing the points $\lambda = \{0, \frac{1}{64}\}$ from $\C$. Transcendental monodromy of the family $\calX_4 \to U_{4}$ has order $2$ around $\lambda = \frac{1}{64}$ and infinite order around $\lambda \in \{0,\infty\}$.

Furthermore, the period map of $\calX_4 \to U_{4}$ is injective and the subgroup of $\mathrm{O}(M_4^\perp)^*$ generated by monodromy transformations is $\mathrm{O}(M_4^\perp)^*$ itself. The family $\calX_4$ thus satisfies Definition \ref{Xndefn}.
\end{proposition}
\begin{proof} As in the proof of Lemma \ref{lemma:M3monodromy}, to see that fibers of $\mathcal{X}_4$ degenerate only when $\lambda \in \{0,\frac{1}{64},\infty\}$, it is enough to do a simple discriminant computation. The elliptic fibration described above is well-defined away from $\lambda \in \{0,\infty\}$ and the two $I_1$ singular fibers collide when $\lambda = \frac{1}{64}$. As before, this shows that monodromy has order $2$ around $\lambda = \frac{1}{64}$.

To see that monodromies around $\lambda \in \{0,\infty\}$ have infinite order, we argue as follows. We have a period map from  $\mathbb{P}^1_\lambda$ to $\overline{\mathcal{M}_{M_4}}$, the Baily-Borel compactification of the period space of $M_4$-polarized K3 surfaces. The monodromy of $\mathcal{X}_4$ is a $(2,k,l)$ triangle group for some choice of $k,l$, and lies inside of $\mathrm{O}(M_4^\perp)^* \cong \Gamma_0(4)^+$ (which is a $(2,\infty,\infty)$ triangle group) as a finite index subgroup, since the period map is dominant. However, by \cite{ccatg}, the only $(2,k,l)$ triangle group of finite index inside of the $(2,\infty,\infty)$ triangle group is the $(2,\infty,\infty)$ triangle group itself (equipped with the identity embedding). Therefore the period map is the identity and monodromy around $\lambda \in \{0,\infty\}$ is of infinite order.
\end{proof}

As in the previous cases, we can use the local form \eqref{eq:M4} of the family $\calX_4$ to find $M_4$-polarized families of K3 surfaces on the threefolds from \cite[Table 1]{doranmorgan}. We find that the generalized functional invariants are given by
\[ \lambda = \frac{Au^{i+j}}{t^i(u-t)^j},\]
where $(t,u)$ are homogeneous coordinates on the base $U \subset \Proj^1$ of the K3 fibration, $(i,j)$ are as in Table \ref{table:12cases}, and $A$ is the Calabi-Yau deformation parameter. This completes the proof of Theorem \ref{thm:Mnfamilies}.

\subsection{The case \texorpdfstring{$n=1$}{n=1}}\label{sect:M1} It remains to address the case of threefolds from \cite[Table 1]{doranmorgan} that are fibred by $M_1$-polarized K3 surfaces. Unfortunately many of the results that we have proved so far do not apply in this case: Assumption \ref{I2ass} does not hold (this follows easily from Remark \ref{rem:assumption} and the expressions for the $(a,b,d)$-parameters of $M_1$-polarized K3 surfaces, below), so the methods of Section \ref{undokummer} do not apply, and the torically induced fibrations of these threefolds by $M_1$-polarized K3 surfaces (computed with \emph{Sage}) cannot all be seen as pull-backs of special $M_1$-polarized families $\calX_1$ from the moduli space $\calM_{M_1}$, so we cannot directly use the results of Section \ref{section:specialfamilies} either.

Instead, we will construct a special $2$-parameter $M_1$-polarized family of K3 surfaces $\calX_{1}^2 \to U_{1}^2$, which is very closely related to a family $\calX_1$ satisfying Definition \ref{Xndefn} (this relationship will be made precise in Proposition \ref{prop:M1inv} and Remark \ref{rem:inv}), and show that the $M_1$-polarized fibrations $\calX \to U$ on our threefolds are pull-backs of this family by maps $U \to U_{1}^2$.

Now let $\calY_{1}^2 \to U_{1}^2$ denote the family of Kummer surfaces associated to $\calX_{1}^2 \to U_{1}^2$ and suppose that we can construct a cover $V \to U_{1}^2$ that undoes the Kummer construction for $\calY_{1}^2$. Then, as before, we may undo the Kummer construction for the family of Kummer surfaces associated to $\calX \to U$ by pulling back to the fibre product $U \times_{U_{1}^2} V$.

To construct the $2$-parameter family $\calX_{1}^2 \to U_{1}^2$, we begin by noting that an $M_1$-polarized K3 surface is mirror to a $\langle 2 \rangle$-polarized K3 surface, which can generically be expressed as a hypersurface of degree $6$ in $\WP(1,1,1,3)$. By the Batyrev-Borisov construction, an $M_1$-polarized K3 surface can be realised torically as an anticanonical hypersurface in the polar dual of $\WP(1,1,1,3)$. The defining polynomial of a generic such anticanonical hypersurface is
\begin{equation}
\label{eq:M1full} a_0 x_0^6 + a_1 x_1^6 + a_2 x_2^6 + a_3 x_3^2 + a_4 x_0 x_1 x_2 x_3 + a_5 x_0^2 x_1^2 x_2^2, 
\end{equation}
where $x_0$, $x_1$, $x_2$ are variables of weight $1$ and $x_3$ is a variable of weight $3$.

On the maximal torus, the family defined by this equation is isomorphic to the vanishing locus in $(\C^{\times})^3$ of the rational polynomial
\begin{equation}\label{eq:M1} y + z + \frac{\alpha}{x^3yz} + x + 1 + \frac{\beta}{x} = 0,\end{equation}
where $\alpha = \frac{a_0 a_1 a_2 a_3^3}{a_4^6}$ and $\beta=\frac{a_3 a_5}{a_4^2}$. Consider the family of K3 surfaces over $\C^2$ obtained by varying $\alpha$ and $\beta$. By resolving the singularities of the generic fibre and removing any singular fibres that remain, we obtain the $2$-parameter family of K3 surfaces $\calX_{1}^2 \to U_{1}^2 \subset \C^2$. 

We can express the $(a, b, d)$-parameters (see Section \ref{M-polarized})  of a fibre of $\calX_{1}^2$ in terms of $\alpha$ and $\beta$ as
\begin{gather*}
a = 1
,\qquad
b = \frac{2^63^3 \alpha}{(4\beta - 1)^3} +1
,\qquad
d = \left(\frac{2^63^3 \alpha}{(4 \beta - 1)^3}\right)^2,
\end{gather*}
where this parameter matching was computed using the elliptic fibrations on $M$-polarized K3 surfaces in Weierstrass normal form. 

Introducing a new parameter
\[ \gamma := \frac{2^63^3 \alpha}{(4 \beta - 1)^3},\]
we see from the expressions for $(a,b,d)$ above that $\gamma$ parametrizes the moduli space $\calM_{M_1}$, so the generalized functional invariant of the family $\calX_{1}^2$ is given by $\gamma$. Then we find:

\begin{lemma} $U^2_{1}$ is the open set $U^2_{1} := \{(\alpha,\beta)\in \C^2 \mid \gamma \notin \{0,-1,\infty\}\}$. Furthermore, $\calX_{1}^2 \to U_{1}^2$ is an $M_1$-polarized family of K3 surfaces.
\end{lemma}
\begin{proof} Using the computer software \emph{Sage}, it is possible to explicitly compute a toric resolution of a generic K3 surface defined in the polar dual of $\WP(1,1,1,3)$ by Equation \eqref{eq:M1full}. From this, we find that the singular fibres of this family occur precisely over $\gamma \in \{0,-1,\infty\}$.

To see that $\calX_{1}^2 \to U_{1}^2$ is an $M_1$-polarized family, we note that $\mathcal{X}_{1}^2$ is a family of hypersurfaces in the polar dual to $\WP(1,1,1,3)$. By \cite{lptk3s}, there is a toric resolution $Y$ of the ambient space such that the fibres $X$  of $\mathcal{X}_{1}^2$ become smooth K3 surfaces in $Y$ and the restriction map
\[ \mathrm{res} : \NS(Y) \rightarrow \NS(X)\]
is surjective. Furthermore the image of $\mathrm{res}$ is the lattice $M_1$. This defines a lattice polarization on each fiber and, since this polarization is induced from the ambient threefold, it is unaffected by monodromy. Thus $\mathcal{X}_{1}^2$ is a family of $M_1$-polarized K3 surfaces. 
\end{proof}

Changing variables in \eqref{eq:M1} and completing the square in $x$, the family $\calX_1^2$ may be written on $(\C^{\times})^3$ as the vanishing locus of
\begin{equation*} \frac{x^2}{4\beta - 1} + y + z + \frac{\gamma}{yz} + 1 = 0.\end{equation*}
Furthermore, we note that points $(\alpha,\beta) \in U_{1}^2$ correspond bijectively with points $(\beta,\gamma)$ in $\{(\beta,\gamma)\in \C^2 \mid \beta \neq \frac{1}{4},\  \gamma \notin \{0,-1\}\}$. Using this we can reparametrize $U_{1}^2$ by $\beta$ and $\gamma$, and thus think of $\calX_1^2 \to U^2_{1}$ as the $2$-parameter family parametrized by $\beta$ and $\gamma$ given on the maximal torus by the expression above.

After performing this reparametrization, the generalized functional invariant map of the family $\calX_{1}^2$ is given simply by projection onto $\gamma$. The fibres of this map are $1$-parameter families of K3 surfaces with the same period, parametrized by $\beta \in \C - \{\frac{1}{4}\}$, which are therefore isotrivial. It is tempting to expect that these isotrivial families are in fact trivial, but this is not the case. Instead, we find:

\begin{proposition} \label{prop:M1inv} Monodromy around the line $\beta = \frac{1}{4}$ fixes the N\'{e}ron-Severi lattice of a generic fibre of $\calX_1^2$ and acts on the transcendental lattice as multiplication by $\mathrm{-Id}$.

Furthermore, the family $\hat{\calX}_1^2$ obtained by pulling back $\mathcal{X}_{1}^2$ to the double cover of $U^2_{1}$ ramified over the line $\beta = \frac{1}{4}$ is isomorphic to a direct product $\calX_1 \times \C^{\times}$, where $\calX_1$ is an $M_1$-polarized family of K3 surfaces satisfying Definition \ref{Xndefn}.
\end{proposition}
\begin{proof} The double cover of $U^2_{1}$ ramified over the line $\beta = \frac{1}{4}$ is given by the map $\C^{\times} \times (\C - \{0,-1\}) \to U_{1}^2$ taking $(\mu,\gamma) \to (\beta,\gamma) = (\mu^2 + \frac{1}{4}, \gamma)$. After a change of variables $x \mapsto x\mu$, the family $\hat{\calX}_1^2$ may be written on the maximal torus $(\C^{\times})^3$ as the vanishing locus of the rational polynomial
\begin{equation} \label{eq:M1gamma} x^2 + y + z + \frac{\gamma}{3^3yz} + 1 = 0.\end{equation}

This family does not depend upon $\mu$, so $\hat{\calX}_1^2$ is isomorphic to a direct product $\calX_1 \times \C^{\times}$, for some family $\calX_1 \to (\C - \{0,-1\})$ parametrized by $\gamma$, and its monodromy around $\mu = 0$ is trivial. Furthermore, for two K3 surfaces $X_1$ and $X_2$ in $\hat{\calX}_1^2$ lying above a fiber $X$ in $\mathcal{X}_{1}^2$ there are natural isomorphisms
\[\phi_1 : X_1 \rightarrow X, \qquad \phi_2 : X_2 \rightarrow X.\]
The automorphism $\phi_1^{-1} \cdot \phi_2$ is the non-symplectic involution given on the maximal torus by $(x,y,z) \mapsto (-x,y,z)$, which fixes the lattice $M_1 = \NS(X)$.

Therefore monodromy around $\beta = 1/4$ has order $2$ and acts on $T_X$ in the same way as a non-symplectic involution $\iota$ with fixed lattice $M_1 = \NS(X)$. Thus, $T_X = (\NS(X)^\iota)^\perp$ and so $\iota$ acts irreducibly on $T_X$ with order $2$. It must therefore act as $-\Id$.

It remains to prove that the $1$-parameter family $\calX_1 \to (\C - \{0,-1\})$ given on the maximal torus by varying $\gamma$ in \eqref{eq:M1gamma} satisfies Definition \ref{Xndefn}. We have already noted that the generalized functional invariant map $(\C - \{0,-1\}) \to \calM_{M_1}$ defined by $\gamma$ is injective. Furthermore, using the expressions for $a$, $b$ and $d$ calculated earlier we see that $\gamma = -1$ at the elliptic point of order $2$, $\gamma = \infty$ at the elliptic point of order $3$, and $\gamma = 0$ at the cusp. All that remains is to check that the monodromy of the family $\calX_1 \to (\C - \{0,-1\})$ has the appropriate orders around each of these points.

This family $\calX_1$ has been studied by Smith \cite[Example 2.15]{pfdefk3s}, where it appears as family $\mathcal{D}$ in Table 2.2 (and we note that Smith's parameter $\mu$ is equal to $-\frac{1}{\gamma}$ in our notation). Its monodromy around the points $\gamma \in \{0,-1,\infty\}$ is given by the symmetric squares of the matrices calculated in \cite[Example 3.9]{pfdefk3s}; in particular we find that this monodromy has the required orders.
\end{proof}

\begin{remark} \label{rem:inv} We note that the complicating factor in the $M_1$-polarized case is the fact that a generic $M_1$-polarized K3 surface $X$ admits a non-symplectic involution which fixes $M_1 \subseteq \NS(X)$. It is this which prevents some of the torically induced fibrations of the threefolds in \cite[Table 1]{doranmorgan} by $M_1$-polarized K3 surfaces from being expressible as pull-backs of an $M_1$-polarized family $\calX_1$ from the moduli space $\calM_{M_1}$. However, from Proposition \ref{prop:M1inv}, we find that we \emph{can} express these fibrations as pull-backs of $\calX_1$ if we proceed to a double cover of the base which kills this involution.
\end{remark}

Given this result, it is easy to undo the Kummer construction for the family $\calY_1^2 \to U_{1}^2$ of Kummer surfaces associated to the family $\calX_1^2$. First, pull back $\calY_1^2$ to the double cover $ (\C -\{0,-1\}) \times \C^{\times} \cong  U_{M_1} \times \C^{\times}$ of $U_{1}^2$ ramified over the line $\beta = \frac{1}{4}$ (where $U_{M_1}$ is defined as in Section \ref{section:specialfamilies}). The result is the family of Kummer surfaces associated to the family $\hat{\calX}_1^2 \cong \calX_1 \times \C^{\times}$. This is exactly the family $\calY_1 \times \C^{\times}$, where $\calY_1 \to U_{M_1}$ is the family of Kummer surfaces associated to $\calX_1$. The Kummer construction can then be undone for this family by pulling back to the cover $V = C_{M_1} \times \C^{\times}$ of $U_{M_1} \times \C^{\times}$, where the cover $C_{M_1} \to U_{M_1}$ is as calculated in Section \ref{section:covers}. 

Thus, given a family $\calX \to U$ of $M_1$-polarized K3 surfaces that can be expressed as the pull-back of the family $\calX_1^2$ by a map $U \to U_1^2$, we may undo the Kummer construction for the associated family of Kummer surfaces $\calY \to U$ by pulling back to the cover $V \times_{U_1^2} U$.

We conclude by applying this to the cases from \cite[Table 1]{doranmorgan}. We find:

\begin{theorem} \label{thm:M1families} There exist K3 fibrations with $M_1$-polarized generic fibre on five of the threefolds in \cite[Table 1]{doranmorgan}, given by the mirrors of those listed in Table \ref{M1tab}.

Furthermore, if $\calX \to \Proj^1_{t,u}$ denotes one of these fibrations and $U \subset \Proj^1_{t,u}$ is the open set over which the fibres of $\calX$ are nonsingular, then the restriction $\calX|_U\to U$ agrees with the pull-back of the family $\calX_1^2$ by the map $U \to U_1^2$ defined by $\alpha$ and $\beta$ in Table \ref{M1tab} \textup{(}in this table $(t,u)$ are coordinates on the base $U \subset \Proj^1_{t,u}$ of the fibration, $A$ is the Calabi-Yau deformation parameter and $k \in \C- \{0,\frac{1}{4}\}$ is a constant\textup{)}. The family $\calX|_U\to U$ is thus an $M_1$-polarized family of K3 surfaces. \end{theorem}

\begin{table}
\makebox[\textwidth]{
\begin{tabular}{|c|c|c|c|}
\hline
Mirror Threefold & $\alpha$ & $\beta$ & $\gamma$ \\ 
\hline
$\WP(1,1,1,1,2)[6]$ & $\dfrac{A (t+u)^3}{ tu^2}$ & $0$ & $-\dfrac{2^63^3 A (t+u)^3}{tu^2}$\\ 
$\WP(1,1,1,1,4)[8]$ & $\dfrac{Au}{t}$ & $\dfrac{t}{u}$ & $\dfrac{2^63^3Au^4}{t(4t-u)^3}$\\ 
$\WP(1,1,1,2,5)[10]$ & $\dfrac{Au^2}{ t^2}$ & $\dfrac{t}{u}$ & $\dfrac{2^63^3Au^5}{t^2(4t-u)^3}$\\
$\WP(1,1,1,1,1,3)[2,6]^*$ & $-\dfrac{Au^2}{t (t + u)}$ & $k$ & $-\dfrac{2^63^3Au^2}{(4k-1)^3t (t + u)}$  \\
$\WP(1,1,1,2,2,3)[4,6]^*$ & $\dfrac{Au^4}{t^2 (t + u)^2}$ & $k$ & $\dfrac{2^63^3Au^4}{(4k-1)^3 t^2 (t + u)^2}$ \\
\hline
\end{tabular}
}
\caption{Values of $\alpha$ and $\beta$ for threefolds admitting $M_1$-polarized fibrations.}
\label{M1tab}
\end{table}

\begin{proof} This is proved in the same way as Theorem \ref{thm:Mnfamilies}, by comparing the forms of the maximal tori in the threefolds from \cite[Table 1]{doranmorgan} to the local form of the family $\calX_1^2$ given by Equation \eqref{eq:M1}.
\end{proof}
%

Finally, we note that the generalized functional invariants in these cases are given by $\gamma$ in Table \ref{M1tab}. We see that, as in Section \ref{sect:14caseapp}, they are all $(i+j)$-fold covers of $\calM_{M_1} \cong \Gamma_0(1) \setminus \HH$ (where $(i,j)$ are as in Table \ref{table:12cases}) having exactly four ramification points: one of order $(i+j)$ over the cusp, two of orders $i$ and $j$ over the elliptic point of order $3$, and one of order $2$ which varies with the value of the Calabi-Yau deformation parameter $A$.

\begin{remark} There is precisely one case from \cite[Table 1]{doranmorgan} that has not been discussed: the mirror of the complete intersection $\WP(1,1,2,2,3,3)[6,6]$. However, it can be seen that this threefold does not admit any torically induced $M$-polarized K3 fibrations, and our methods have not yielded any that are not torically induced either.
\end{remark}

\section{Application to the arithmetic/thin dichotomy}\label{sect:arith/thin}

Recall that each of the threefolds $X$ from \cite[Table 1]{doranmorgan} moves in a one parameter family over the thrice-punctured sphere $\Proj^1-\{0,1,\infty\}$. Recently there has been a great deal of interest in studying the action of monodromy around the punctures on the third cohomology $H^3(X,\Z)$. This monodromy action defines a Zariski dense subgroup of $\mathrm{Sp}(4,\R)$, which may be either arithmetic or non-arithmetic (more commonly called \emph{thin}). Singh and Venkataramana \cite{acshg}\cite{a4mgacyt} have proved that the monodromy is arithmetic in seven of the fourteen cases from \cite[Table 1]{doranmorgan}, and Brav and Thomas \cite{tmsp4} have proved that it is thin in the remaining seven. The arithmetic/thin status of each of the threefolds from Theorems \ref{thm:Mnfamilies} and \ref{thm:M1families} is given in the fifth column of Table \ref{table:12cases}.

It is an open problem to explain this behaviour geometrically. To this end, we are able to make an interesting observation concerning the arithmetic/thin dichotomy for the $M_n$-polarized families with Theorems \ref{thm:Mnfamilies} and \ref{thm:M1families}. Specifically, from Table \ref{table:12cases} we observe that a threefold admitting a torically induced fibration by $M_n$-polarized K3 surfaces has thin monodromy if and only if neither of the values $(i,j)$ associated to this fibration are equal to $2$.

This observation may also be extended to the 14th case \cite{14thcase}. In this case, recall that the threefold $Y_1$, which moves in a one-parameter family realising the 14th case variation of Hodge structure, admits a torically induced fibration by $M$-polarized K3's rather than $M_n$-polarized K3's. Thus the generalized functional invariant map from $Y_1$ has image in the $2$-dimensional moduli space of $M$-polarized K3 surfaces, rather than one of the modular curves $\calM_{M_n}$. However, from \cite[Section 5.1 and Equation (4.5)]{14thcase}, we see that the image of the generalized functional invariant map from $Y_1$ is contained in the special curve in the $M$-polarized moduli space defined by the equation $\sigma = 1$ (where $\sigma$ and $\pi$ are the rational functions from Section \ref{M-polarized}). 

By the results of \cite[Section 3.1]{nfk3smmp}, the moduli space of $M$-polarized K3 surfaces may be identified with the Hilbert modular surface
\[(\mathrm{PSL}(2,\Z)\times \mathrm{PSL}(2,\Z)) \rtimes \Z/2\Z \, \setminus \, \HH \times \HH,\]
with natural coordinates given by $\sigma$ and $\pi$. The $\sigma = 1$ locus is thus parametrized by $\pi$ and has an orbifold structure induced from the Hilbert modular surface. This orbifold structure has an elliptic point of order six at $\pi = 0$, an elliptic point of order two at $\pi = \frac{1}{4}$, and a cusp at $\pi = \infty$.

The generalized functional invariant map for the K3 fibration on $Y_1$ is given by the rational function $\pi$, which is calculated explicitly in \cite[Equation (4.4)]{14thcase}. It is a double cover of the $\sigma = 1$ locus ramified over the cusp and a second point that varies with the value of the Calabi-Yau deformation parameter. This agrees perfectly with the description of the generalized functional invariants for the $M_n$-polarized cases from Section \ref{sect:14caseapp}, with $(i,j) = (1,1)$, thereby giving the final row of Table \ref{table:12cases}. From this table, we observe:

\begin{theorem} \label{thm:thinarith} Suppose that $\calX$ is a family of Calabi-Yau threefolds from \cite[Table 1]{doranmorgan} that admit a torically induced fibration by $M_n$-polarized K3 surfaces \textup{(}resp. $M$-polarized K3 surfaces with $\sigma = 1$\textup{)}. By our previous discussion, the generalized functional invariant of this fibration is a $(i+j)$-fold cover of the modular curve $\calM_{M_n} \cong \Gamma_0(n)^+ \setminus \HH$ \textup{(}resp. the orbifold curve given by the $\sigma = 1$ locus in the moduli space of $M$-polarized K3 surfaces\textup{)}, where $i$ and $j$ are given by Table \ref{table:12cases}, which is totally ramified over the cusp and ramified to orders $i$ and $j$ over the remaining orbifold point of order $\neq 2$. Then $\calX$ has thin monodromy if and only if neither $i$ nor $j$ is equal to $2$.
\end{theorem}

\bibliography{Books}
\bibliographystyle{amsalpha}
\end{document}